\documentclass[11pt]{article}

\usepackage{amsmath,amssymb,amsthm}
\usepackage{fullpage}
\usepackage{authblk}
\usepackage{graphicx}
\usepackage{algorithm}
\usepackage[noend]{algorithmic}

\title{Beyond Helly graphs: the diameter problem on absolute retracts}
\author[1,2]{Guillaume Ducoffe}
\affil[1]{\small National Institute for Research and Development in Informatics, Romania}
\affil[2]{\small University of Bucharest, Romania}
\date{}

\newtheorem{lemma}{Lemma}
\newtheorem{theorem}{Theorem}
\newtheorem{corollary}{Corollary}
\newtheorem{proposition}{Proposition}
\newtheorem{myclaim}{Claim}

\newcommand{\qedclaim}{\hfill $\diamond$ \medskip}
\newenvironment{proofclaim}{\noindent\ignorespaces{\em Proof.}}{\hfill\qedclaim\par\noindent} 

\begin{document}

\maketitle

\begin{abstract}
Characterizing the graph classes such that, on $n$-vertex $m$-edge graphs in the class, we can compute the diameter faster than in ${\cal O}(nm)$ time is an important research problem both in theory and in practice. 
We here make a new step in this direction, for some metrically defined graph classes. Specifically, a subgraph $H$ of a graph $G$ is called a retract of $G$ if it is the image of some idempotent endomorphism of $G$. Two necessary conditions for $H$ being a retract of $G$ is to have $H$ is an isometric and isochromatic subgraph of $G$. We say that $H$ is an absolute retract of some graph class ${\cal C}$ if it is a retract of any $G \in {\cal C}$ of which it is an isochromatic and isometric subgraph. In this paper, we study the complexity of computing the diameter within the absolute retracts of various hereditary graph classes. First, we show how to compute the diameter within absolute retracts of bipartite graphs in randomized $\tilde{\cal O}(m\sqrt{n})$ time. For the special case of chordal bipartite graphs, it can be improved to linear time, and the algorithm even computes all the eccentricities. Then, we generalize these results to the absolute retracts of $k$-chromatic graphs, for every fixed $k \geq 3$. Finally, we study the diameter problem within the absolute retracts of planar graphs and split graphs, respectively.    
\end{abstract}

%keywords: absolute retract; chordal bipartite graphs; split graphs; planar graphs; diameter computation.

\section{Introduction}\label{sec:intro}

One of the most basic graph properties is the diameter of a graph (maximum number of edges on a shortest path). It is a rough estimate of the maximum delay in order to send a message in a communication network~\cite{DeR94}, but it also got used in the literature for various other purposes~\cite{AJB99,WaS98}. The complexity of computing the diameter has received tremendous attention in the Graph Theory community~\cite{AVW16,BCT17,BHM20,Cab18,CDV02,CDEHV08,Che96,CDDH+01,CDP19,Dam16,Dra99,DuD19+,DHV19,DHV20,Duc19,FaP80,GKHM+18,Ola90}. Indeed, while this can be done in ${\cal O}(nm)$ time for any $n$-vertex $m$-edge graph, via a simple reduction to breadth-first search, breaking this quadratic barrier (in the size $n+m$ of the input) happens to be a challenging task. In fact, under plausible complexity assumptions such as the Strong Exponential-Time Hypothesis (SETH), the optimal running time for computing the diameter is essentially in ${\cal O}(nm)$ --- up to sub-polynomial factors~\cite{RoV13}. This negative result holds even if we restrict ourselves to bipartite graphs or split graphs~\cite{AVW16,BCH16}. However, on the positive side, several recent works have characterized important graph classes for which we can achieve for the diameter problem ${\cal O}(m^{2-\epsilon})$ time, or even better ${\cal O}(mn^{1-\epsilon})$ time, for some $\epsilon > 0$. Next, we focus on a few such classes that are most relevant to our work. Specifically, we call $G=(V,E)$ a Helly graph if every family of pairwise intersecting balls of $G$ (of arbitrary radius and center) have a nonempty common intersection. The Helly graphs are a broad generalization of many better-known graph classes, such as: trees, interval graphs, strongly chordal graphs and dually chordal graphs~\cite{BaC08}. Furthermore, a celebrated theorem in Metric Graph Theory is that every graph is an isometric (distance-preserving) subgraph of some Helly graph~\cite{Dre84,Isb64}. Other properties of Helly graphs were also thoroughly investigated in prior works~\cite{BaP89,BaP89b,BaP91,CLP00,Dra89,Dra93,DrB96,DrG19,LiS07,Pol01,Pol03}. In particular, as far as we are concerned here, there is a randomized $\tilde{\cal O}(m\sqrt{n})$-time algorithm in order to compute the diameter within $n$-vertex $m$-edge Helly graphs with high probability~\cite{DuD19+}.

Larger classes, related to the Helly graphs, have been considered recently. For instance, $G=(V,E)$ is a $k$-Helly graph if every family of $k$-wise intersecting balls of $G$ have a nonempty common intersection (Helly graphs are exactly the $2$-Helly graphs). For every fixed $k$, there is a randomized $\tilde{\cal O}(m\sqrt{n})$-time algorithm in order to compute the radius (minimum eccentricity of a vertex) within $k$-Helly graphs~\cite{Duc20}. The Helly-gap of $G=(V,E)$ is the least $\alpha$ such that, for every family of pairwise intersecting balls of $G$, if we increase all the radii by $\alpha$ then this family has a nonempty common intersection~\cite{CCGH+20,DrG020}. It also follows from~\cite{DuD19+} that the radius and the diameter of a graph with bounded Helly-gap can be approximated up to some additive constant, that only depends on its Helly-gap. The latter result generalizes prior work on diameter and center approximations within hyperbolic graph classes~\cite{CDEHV08}. Finally, the graphs of bounded ``distance VC-dimension'' were introduced in~\cite{CEV07}, where it was observed that, by a result from~\cite{Mat04}, they satisfy certain ``fractional'' Helly property. Many interesting graph classes have bounded distance VC-dimension, such as: proper minor-closed graph classes~\cite{CEV07}, interval graphs~\cite{DHV20} and bounded clique-width graphs~\cite{BoT15}. For all the aforementioned sub-classes, there exist algorithms in ${\cal O}(mn^{1-\epsilon})$ time, for some $\epsilon > 0$, in order to compute all the eccentricities, and so, the diameter~\cite{Duc20cw,DHV20,Ola90}. Partial results also have been obtained for the diameter problem on all graph classes of bounded distance VC-dimension~\cite{DHV20,DuD19+}.

Recall that an endomorphism of a graph $G$ is an edge-preserving mapping of $G$ to itself. A retraction is an idempotent endomorphism. If $H$ is the image of $G$ by some retraction (in particular, $H$ is a subgraph of $G$) then, we call $H$ a retract of $G$. The notion of retract has applications in some discrete facility location problems~\cite{Hel74}, and it is useful in characterizing some important graph classes. For instance, the median graphs are exactly the retracts of hypercubes~\cite{Ban84}. We here focus on the relation between retracts and Helly graphs, that is as follows. For some class ${\cal C}$ of reflexive graphs ({\it i.e.}, with a loop at every vertex), let us define the absolute retracts of ${\cal C}$ as those $H$ such that, whenever $H$ is an isometric subgraph of some $G \in {\cal C}$, $H$ is a retract of $G$. Absolute retracts find their root in Geometry, where they got studied for various metric spaces~\cite{Kli82}. In the special case of the class of all reflexive graphs, the absolute retracts are exactly the Helly (reflexive) graphs~\cite{HeR87}. Motivated by this characterization of Helly graphs, and the results obtained in~\cite{DuD19+} for the diameter problem on this graph class, we here consider the following notion of absolute retracts, for irreflexive graphs. -- Unless stated otherwise, all graphs considered in this paper are irreflexive. -- Namely, let us first recall that a subgraph $H$ of a graph $G$ is isochromatic if it has the same chromatic number as $G$. Then, given a class of (irreflexive) graphs ${\cal C}$, the absolute retracts of ${\cal C}$ are those $H$ such that, whenever $H$ is an isometric and isochromatic subgraph of some $G \in {\cal C}$, $H$ is a retract of $G$. We refer the reader to~\cite{BDS87,BFH93,BaP99,Hel74,Hel74a,Kla94,KlW18,Lot10,PeP85,Pes87,Pes88}, where this notion got studied for various graph classes.

\paragraph{Our results.} In this paper, we prove new structural and algorithmic properties of the absolute retracts of various hereditary graph classes, such as: bipartite graphs, $k$-chromatic graphs (for any $k \geq 3$), split graphs and planar graphs. Our focus is about the diameter problem on these graph classes but, on our way, we uncover several nice properties of the shortest-path distribution of their absolute retracts, that may be of independent interest. 
\begin{itemize}
\item First, in Sec.~\ref{sec:k=2}, we consider the absolute retracts of bipartite graphs and some important subclasses of the latter. We observe that in the square of such graph $G$, its two partite sets induce Helly graphs. This result complements the known relations between Helly graphs and absolute retracts of bipartite graphs~\cite{BFH93}. Then, we show how to compute the diameter of $G$ from the diameter of both Helly graphs (actually, from the knowledge of the peripheral vertices in these graphs, {\it i.e.}, those vertices with maximal eccentricity). Recently~\cite{DDG19+}, we announced an ${\cal O}(m\sqrt{n})$-time algorithm in order to compute all the eccentricities in a Helly graph. However, extending this result to the absolute retracts of bipartite graphs appears to be a more challenging task. We manage to do so for the subclass of chordal bipartite graphs, for which we achieve a linear-time algorithm in order to compute all the eccentricities. For that, we prove the stronger result that in the square of such graph, its two partite sets induce strongly chordal graphs. Here also, our result complements the known relations between both graph classes~\cite{Bra91,FaM83}. 
\item In Sec.~\ref{sec:k-chromatic}, we generalize our above framework to the absolute retracts of $k$-chromatic graphs, for any $k \geq 3$. Our proofs in this part are more technical and intricate than in Sec.~\ref{sec:k=2}. For instance, we cannot extract a Helly graph from each colour class anymore. Instead, we define a partial eccentricity function for each colour ({\it i.e.}, by restricting ourselves to the distances between vertices of the same colour), and we prove that the latter functions almost have the same properties as the eccentricity function of a Helly graph. 
%For that, we heavily use the property that, for every $k \geq 4$, between every two vertices of the same colour that are at even distance one from the other, there always exists a bichromatic shortest-path. 
%For $k=3$, our proofs are even more complex due to the need to consider some special type of shortest-paths with alternating three colours. It results in a blow-up of the running time in order to compute the diameter: from $\tilde{\cal O}(m\sqrt{n})$ for $k=2$ and $k \geq 4$ to $\tilde{\cal O}(mn^{2/3})$ for $k=3$. 
\item Our positive results in Sec.~\ref{sec:k=2} and~\ref{sec:k-chromatic} rely on some Helly-type properties of the graph classes considered. However, our hardness result in Sec.~\ref{sec:split} hints that the weaker property of being an absolute retract of some well-structured graph class is not sufficient on its own for faster diameter computation. Specifically, we prove that under SETH, there is no ${\cal O}(mn^{1-\epsilon})$-time algorithm for the diameter problem, for any $\epsilon > 0$, on the class of absolute retracts of split graphs. This negative result follows from an elegant characterization of this subclass of split graphs in~\cite{Kla94}.
\item Finally, in Sec.~\ref{sec:planar}, we briefly consider the absolute retracts of planar graphs. While there now exist several truly subquadratic-time algorithms for the diameter problem on all planar graphs~\cite{Cab18,DHV20,GKHM+18} -- with the best-known running time being in $\tilde{\cal O}(n^{5/3})$ -- the existence of a quasi linear-time algorithm for this problem has remained so far elusive, and it is sometimes conjectured that no such algorithm exists~\cite{Cab18}. We give evidence that finding such algorithm for the absolute retracts of planar graphs is already a hard problem on its own. Specifically, we prove that every planar graph is an isometric subgraph of some absolute retract of planar graphs. This result mirrors the aforementioned property that every graph isometrically embeds in a Helly graph~\cite{Dre84,Isb64}. It implies the existence of some absolute retracts of planar graphs with treewidth arbitrarily large and inner vertices of degree three. Doing so, we rule out two general frameworks in order to compute the diameter in quasi linear time on some subclasses of planar graphs~\cite{CDV02,Epp02}.  
\end{itemize}
Let us mention that all graph classes considered here are polynomial-time recognizable. For the absolute retracts of $k$-chromatic graphs, the best-known recognition algorithms have superquadratic running-time (even for $k=2$)~\cite{BDS87,BaP99}. Fortunately, we do not need to execute these recognition algorithms before we can compute the diameter of these graphs. Indeed, our algorithms in Sec.~\ref{sec:k=2} and~\ref{sec:k-chromatic} are heuristics which can be applied to any graph. Sometimes, these algorithms may fail in outputting a value, that certifies the input graph is not an absolute retract of $k$-chromatic graphs, for some $k$. Conversely, if the input graph is an absolute retract of $k$-chromatic graphs, then  the algorithm always succeeds in outputting a value and this value is exactly the diameter. We consider our framework to be especially useful for all subclasses of the absolute retracts of $k$-chromatic graphs that admit quasi linear-time recognition algorithms. In this respect, we stress that the absolute retract of bipartite graphs are a superclass of cube-free modular graphs, and so, of chordal bipartite graphs, cube-free median graphs and covering graphs of modular lattices of breadth at most two~\cite{BDS87}. 

\paragraph{Notations.} We mostly follow the graph terminology from~\cite{BoM08,Die10}. All graphs considered are finite, simple ({\it i.e.}, without loops nor multiple edges), unweighted and connected. For a graph $G=(V,E)$, let the (open) neighbourhood of a vertex $v$ be defined as $N_G(v) = \{ u \in V \mid uv \in E \}$ and its closed neighbourhood as $N_G[v] = N_G(v) \cup \{v\}$. Similarly, for a vertex-subset $S \subseteq V$, let $N_G(S) = \bigcup_{v \in S} N_G(v) \setminus S$, and let $N_G[S] = N_G(S) \cup S$. The distance between two vertices $u,v \in V$ equals the minimum number of edges on a $uv$-path, and it is denoted $d_G(u,v)$. We also let $I_G(u,v)$ denote the vertices ``metrically'' between $u$ and $v$, {\it i.e.}, $I_G(u,v) = \{ w \in V \mid d_G(u,v) = d_G(u,w) + d_G(w,v) \}$. The ball of center $v$ and radius $r$ is defined as $N_G^r[v] = \{ u \in V \mid d_G(u,v) \leq r \}$. Furthermore, let the eccentricity of a vertex $v$ be defined as $e_G(v) = \max_{u \in V} d_G(u,v)$. The diameter and the radius of a graph $G$ are defined as $diam(G) = \max_{v \in V} e_G(v)$ and $rad(G) = \min_{v \in V} e_G(v)$, respectively. A vertex $v \in V$ is called central if $e_G(v) = rad(G)$, and peripheral if $e_G(v) = diam(G)$. Note that we sometimes omit the subscript if the graph $G$ is clear from the context. We introduce additional terminology where it is needed throughout the paper.  

\section{Bipartite graphs}\label{sec:k=2}

The study of the absolute retracts of bipartite graphs dates back from Hell~\cite{Hel72}, and since then many characterizations of this graph class were proposed~\cite{BDS87}. This section is devoted to the diameter problem on this graph class. In Sec.~\ref{sec:k2-diam}, we propose a randomized $\tilde{\cal O}(m\sqrt{n})$-time algorithm for this problem. Then, we consider the chordal bipartite graphs in Sec.~\ref{sec:chordal}, that have been proved in~\cite{BDS87} to be a subclass of the absolute retracts of bipartite graphs. For the chordal bipartite graphs, we present a deterministic linear-time algorithm in order to compute all the eccentricities.

Let us introduce a few additional terminology.
For a connected bipartite graph $G$, we denote its two partite sets by $V_0$ and $V_1$.
A half-ball is the intersection of a ball with one of the two partite sets of $G$.
Finally, for $i \in \{0,1\}$, let $H_i$ be the graph with vertex-set $V_i$ and an edge between every two vertices with a common neighbour in $G$.

\subsection{Faster diameter computation}\label{sec:k2-diam} 

We start with the following characterization of the absolute retracts of bipartite graphs:

\begin{theorem}[\cite{BDS87}]\label{thm:bipartite-retract}
$G=(V,E)$ is an absolute retract of bipartite graphs if and only if the collection of half-balls of $G$ satisfies the Helly property.
\end{theorem}

This above Theorem~\ref{thm:bipartite-retract} leads us to the following simple, but important for what follows, observation about the internal structure of the absolute retracts of bipartite graphs:

\begin{lemma}\label{lem:half-helly}
If $G=(V_0 \cup V_1,E)$ is an absolute retract of bipartite graphs then both $H_0$ and $H_1$ are Helly graphs. 
\end{lemma}

\begin{proof}
For each $i \in \{0,1\}$ the balls of $H_i$ are exactly the half-balls of $G$ that intersect $V_i$ and have as their center a vertex of $V_i$. Therefore, by Theorem~\ref{thm:bipartite-retract}, the collection of balls of $H_i$ satisfies the Helly property, {\it i.e.}, $H_i$ is a Helly graph. 
\end{proof}

Next, we prove that in order to compute $diam(G)$, with $G$ an absolute retract of bipartite graphs, it is sufficient to compute the peripheral vertices of the Helly graphs $H_0$ and $H_1$.

\begin{lemma}\label{lem:diam-bipartite}
If $G=(V_0 \cup V_1,E)$ is an absolute retract of bipartite graphs such that $diam(H_0) \leq diam(H_1)$ then, $diam(G) \in \{2diam(H_1),2diam(H_1)+1\}$. Moreover, if $diam(G) \geq 3$ then we have $diam(G) = 2diam(H_1)+1$ if and only if: 
\begin{itemize}
\item $diam(H_1) = 1$;
\item or $diam(H_0) = diam(H_1)$ and, for some $i \in \{0,1\}$, there exists a peripheral vertex of $H_i$ whose all neighbours in $G$ are peripheral vertices of $H_{1-i}$.   
\end{itemize}
\end{lemma}

\begin{proof}
Let $i \in \{0,1\}$. Clearly, $diam(G) \geq \max_{u,v \in V_i} d_G(u,v) = 2diam(H_i)$. Furthermore, since $G$ is connected, $V_i$ is a dominating set of $G$, and therefore, every vertex of $V_i$ has eccentricity at most $2diam(H_i)+1$. Overall, since we assume $diam(H_1) \geq diam(H_0)$, we get as desired $diam(G) \in \{2diam(H_1),2diam(H_1)+1\}$. Note that this result holds for any bipartite graph.

In what follows, we further assume $diam(G) \geq 3$. If $diam(H_1) \leq 1$ then, $diam(G) \leq 2diam(H_1) + 1 \leq 3$. Therefore, $diam(H_1) = 1$ and $diam(G) = 2 diam(H_1) + 1$. From now on, $diam(H_1) \geq 2$.

Let us first assume $diam(G) = 2diam(H_1)+1$. Observe that for each $i \in \{0,1\}$, every vertex of $V_i$ is at a distance at most $1 + 2diam(H_{1-i})$ from every vertex of $V_{1-i}$ (in order to see that, just take any neighbour of this vertex in $V_{1-i}$). In particular, if $diam(H_0) < diam(H_1)$, then every vertex of $V_0$ has eccentricity at most $2diam(H_0) + 1 < 2diam(H_1)$, while every vertex of $V_1$ has eccentricity at most $\max\{2diam(H_1),2diam(H_0)+1\} = 2diam(H_1)$. Hence, in order to have $diam(G) = 2diam(H_1)+1$, we must have $diam(H_0) = diam(H_1)$. Furthermore, for some fixed $i \in \{0,1\}$, let $v \in V_i$ be a peripheral vertex of $G$. Since $2diam(H_1)+1=e_G(v) \leq 2e_{H_i}(v) + 1$, we must have $e_{H_i}(v) = diam(H_i) = diam(H_1)$. In the same way, for any neighbour $u \in N_G(v)$, we have $e_G(v) \leq \max\{2diam(H_i),2e_{H_{1-i}}(u)+1\}$, and therefore, we must also have $e_{H_{1-i}}(u) = diam(H_{1-i}) = diam(H_1)$. It implies the existence of a peripheral vertex of $H_i$ whose all neighbours are peripheral vertices of $H_{1-i}$. 

Conversely, let us assume that $diam(H_0) = diam(H_1)$ and that, for some $i \in \{0,1\}$, there exists a $v \in V_i$ such that: $e_{H_i}(v) = diam(H_i) = diam(H_1)$; for every $u \in N_G(v)$, $e_{H_{1-i}}(u) = diam(H_{1-i}) = diam(H_1)$. Suppose by contradiction $e_G(v) < 2diam(H_1)+1$. Then, since $G$ is bipartite, all the vertices of $V_{1-i}$ are at a distance at most $2diam(H_1)-1$ from vertex $v$. Furthermore, for every $x,y \in V_{1-i}$, since we have $d_G(x,y) = 2d_{H_{1-i}}(x,y)$, we obtain that for every even $\ell \geq d_{H_{1-i}}(x,y)$ the half-balls $N^{\ell}_G[x] \cap V_{1-i}$ and $N^{\ell}_G[y] \cap V_{1-i}$ intersect. In particular, we may choose $\ell = 2(diam(H_1)-1)$ because $diam(H_1) \geq 2$ implies $2(diam(H_1)-1) \geq diam(H_1) \geq d_{H_{1-i}}(x,y)$. But then, the half-balls $N_G(v)$ ($=N_G^1[v] \cap V_{1-i}$) and $N_G^{2(diam(H_1)-1)}[w] \cap V_{1-i}$, for every $w \in V_{1-i}$, pairwise intersect. By Theorem~\ref{thm:bipartite-retract}, there exists a $u \in N_G(v)$ s.t. $e_{H_{1-i}}(u) \leq diam(H_1)-1$. This is a contradiction because such neighbour $u$ cannot be peripheral in $H_{1-i}$.    
\end{proof}

The remaining of Sec.~\ref{sec:k2-diam} is devoted to the computation of all the peripheral vertices in both Helly graphs $H_0$ and $H_1$. While there exists a truly subquadratic-time algorithm for computing the diameter of a Helly graph~\cite{DuD19+}, we observe that in general, we cannot compute $H_0$ and $H_1$ in truly subquadratic time from $G$. Next, we adapt~\cite[Theorem 2]{DuD19+}, for the Helly graphs, to our needs. 

\begin{lemma}\label{lem:cst-ecc}
If $G=(V_0 \cup V_1,E)$ is an absolute retract of bipartite graphs then, for any $k$, we can compute in ${\cal O}(km)$ time the set of vertices of eccentricity at most $k$ in $H_0$ (resp., in $H_1$).
\end{lemma}

\begin{proof}
By symmetry, we only need to prove the result for $H_0$. Let $U = \{ v \in V_0 \mid e_{H_0}(v) \leq k \}$ be the set to be computed.
We consider the more general problem of computing, for any $t$, a partition ${\cal P}_t = (A^t_1,A^t_2,\ldots,A^t_{p_t})$ of $V_0$, in an arbitrary number $p_t$ of subsets, subject to the following constraints: 
\begin{itemize}
\item For every $1 \leq i \leq p_t$, let $C^t_i := \bigcap_{v \in A^t_i} N^t_{G}[v]$. Let $B_i^t := C^t_i \cap V_0$ if $t$ is even and let $B_i^t := C^t_i \cap V_1$ if $t$ is odd (for short, $B_i^t = C_i^t \cap V_{t \pmod 2}$). \emph{We impose the sets $B_i^t$ to be \underline{nonempty} and \underline{pairwise disjoint}}.
\end{itemize}
Indeed, under these two conditions above, we have $U \neq \emptyset$ if and only if, for any partition ${\cal P}_{2k}$ as described above, $p_{2k} = 1$. Furthermore if it is the case then $U = B^{2k}_1$. 

\smallskip
{\bf The algorithm.}
We construct the desired partition by induction over $t$. If $t = 0$ then, let $V_0 = \{v_1,v_2,\ldots,v_{p_0}\}$. We just set ${\cal P}_0 = (\{v_0\},\{v_1\},\ldots,\{v_{p_0}\})$ (each set is a singleton), and for every $1 \leq i \leq p_0$ let $B_i^0 = A_i^0 = \{v_i\}$. Else, we construct ${\cal P}_t$ from ${\cal P}_{t-1}$. Specifically, for every $1 \leq i \leq p_{t-1}$, we let $W_i^t := N_G(B_i^{t-1})$. Then, starting from $j:=0$ and ${\cal F} := {\cal P}_{t-1}$, we proceed as follows until we have ${\cal F} = \emptyset$. We pick a vertex $u$ s.t. $\#\{ i \mid A_i^{t-1} \in {\cal F}, \ u \in W_i^t \}$ is maximized. Then, we set $A_j^t := \bigcup \{ A_i^{t-1} \mid  A_i^{t-1} \in {\cal F}, \ u \in W_i^t \}$ and $B_j^t := \bigcap \{ W_i^t \mid A_i^{t-1} \in {\cal F}, \ u \in W_i^t \}$. We add the new subset $A_j^t$ to ${\cal P}_t$, we remove all the subsets $A_i^{t-1}, u \in W_i^t$ from ${\cal F}$, then we set $j := j+1$. 

\smallskip
{\bf Correctness.} The base case of our above induction is trivially correct. In order to prove correctness of our inductive step, we need the following two intermediate claims.
\begin{myclaim}
$W_i^t = V_{t \pmod 2} \cap \left( \bigcap_{v \in A^{t-1}_i} N^t_{G}[v] \right)$.
\end{myclaim}
\begin{proofclaim}
Recall that $B_i^{t-1} = V_{ t - 1 \pmod 2} \cap \left( \bigcap_{v \in A^{t-1}_i} N^{t-1}_{G}[v] \right)$, and that $W_i^t = N_G(B_i^{t-1})$.
Therefore by construction, $W_i^t \subseteq V_{t \pmod 2} \cap \left( \bigcap_{v \in A^{t-1}_i} N^t_{G}[v] \right)$. Conversely, let $w \in V_{t \pmod 2} \cap \left( \bigcap_{v \in A^{t-1}_i} N^t_{G}[v] \right)$ be arbitrary. Let $a \in A_i^{t-1}$ be arbitrary. There are two cases. If $a \neq w$ then, by considering any $x \in N_G(w) \cap I_G(a,w)$, we get that the half-balls $N_G(w)$ ($=N_G^1[w] \cap V_{t-1 \pmod 2}$) and $N_G^{t-1}[a] \cap V_{t-1 \pmod 2}$ intersect. Otherwise, $a = w$ and then, $t$ is even. Here also the half-balls $N_G(w)$ and $N_G^{t-1}[a] \cap V_{t-1 \pmod 2}$ ($= N_G^{t-1}[a] \cap V_1$) intersect because we have $N_G(w) = N_G(a) \subseteq N_G^{t-1}[a] \cap V_{t-1 \pmod 2}$. Overall, in all the cases, the half-balls $N_G(w)$ and $N_G^{t-1}[a] \cap V_{t-1 \pmod 2}$, for every $a \in A_i^{t-1}$, pairwise intersect. Then, by Theorem~\ref{thm:bipartite-retract}, vertex $w$ has a neighbour in $B_i^{t-1}$.  
\end{proofclaim}
It follows from this above claim that, for each subset $A_j^t$ created at step $t$, we have $B_j^t = V_{t \pmod 2} \cap \left( \bigcap_{v \in A^{t}_j} N^t_{G}[v] \right)$, as desired. Observe that all the subsets $B_j^t$ are nonempty since they at least contain the vertex $u \in V_{t \pmod 2}$ that is selected in order to create $A_j^t$. What now remains to prove is that all the subsets $B_j^t$ are pairwise disjoint. The latter easily follows from our next intermediate claim, namely: 
\begin{myclaim}\label{claim:k=2}
Let $u \in V_{t \pmod 2}$ be a vertex maximizing $\#\{ i \mid u \in W_i^t \}$. For every index $i'$ s.t. $u \notin W_{i'}^t$, we have $W_{i'}^t \cap \left( \bigcap_{u \in W_i^t} W_i^t \right) = \emptyset$. 
\end{myclaim}
\begin{proofclaim}
It directly follows from the maximality of $\#\{ i \mid u \in W_i^t \}$.
\end{proofclaim}
We are done applying this above claim at each creation of a new subset $A_j^t$. 

\smallskip
{\bf Complexity.} The base case of our induction requires ${\cal O}(n)$ time. Let us prove that each inductive step requires ${\cal O}(m)$ time. First, since by the hypothesis the sets $B_i^{t-1}$ are pairwise disjoint, we can compute the sets $W_i^t$ in total linear time.  Then, we create an array of $p_{t-1}$ lists, numbered from $1$ to $p_{t-1}$. For each vertex $u \in V_{t \pmod 2}$ s.t. $\{ i \mid u \in W_i^t \} \neq \emptyset$, we put it in the list numbered $\#\{ i \mid u \in W_i^t \}$. Since it only requires to scan the $W_i^t$'s once, it takes ${\cal O}(m+n)$ time. We scan the lists in decreasing order ({\it i.e.}, from $p_{t-1}$ downto $1$), going to the next list each time the current one is empty. When the current list is nonempty, we pick any vertex $u$ of this list in order to create the next subset $A_j^t$. Note that both subsets $A_j^t$ and $B_j^t$ can be computed in ${\cal O}(\sum \{|A_i^{t-1}| + |W_i^t| \mid A_i^{t-1} \in {\cal F}, \ u \in W_i^t \})$ time. Finally, we need to discard the vertices of $B_j^t$ from the list in which they are currently contained, while for every other $w \in \left( \bigcup \{ W_i^t \mid A_i^{t-1} \in {\cal F}, u \in W_i^t \} \right) \setminus B_j^t$ we need to update $\#\{ i' \mid A_{i'}^{t-1} \in {\cal F}, w \in W_{i'}^t \}$ and the corresponding list. If, for each vertex, we store a pointer to its position in the corresponding list, we can also do the latter in ${\cal O}(\sum \{|W_i^t| \mid A_i^{t-1} \in {\cal F}, \ u \in W_i^t \})$ time. Overall, since each set $A_i^{t-1}$ gets removed once from ${\cal F}$, the total running time for the inductive step is in ${\cal O}(m+n)$.  
\end{proof}

We use this above Lemma~\ref{lem:cst-ecc} when the respective diameters of $H_0$ and $H_1$ are in ${\cal O}(\sqrt{n})$. For larger values of diameters, we use a randomized procedure (Algorithm~\ref{alg:large-diam}).  

\begin{algorithm}
	\caption{Diameter computation in Helly graphs.}
	\label{alg:large-diam}
	\footnotesize
	
	\begin{algorithmic}[1]
		\REQUIRE{A Helly graph $H$ s.t. $diam(H) > 3k = \omega(\log{|V(H)|})$.}\\
		\STATE{Set $p = c \frac{\log{|V(H)|}}{k}$, for some sufficiently large constant $c$.}
		\STATE{Let $U(p)$ contain every $v \in V(H)$ independently with probability $p$.}
		\FORALL{$v \in V(H)$}
		\IF{$\forall u \in U(p), \ d_H(u,v) > k$}
		\STATE Set $\bar{e}(v) := 0$.
		\ELSE
		\STATE Set $\bar{e}(v) := \min \{ d_H(u,v) + e_H(u) \mid u \in U(p), \ d_H(u,v) \leq k\}$.
		\ENDIF
		\ENDFOR
	\end{algorithmic}
\end{algorithm}

\begin{lemma}[Theorem 3 in~\cite{DuD19+}]\label{lem:large-ecc}
With high probability, Algorithm~\ref{alg:large-diam} runs in $\tilde{\cal O}(|E(H)|\cdot |V(H)|/k)$ time, we have $diam(H) = \max_{v \in V(H)} \bar{e}(v)$, and the peripheral vertices of $H$ are exactly the vertices $v \in V(H)$ which maximize $\bar{e}(v)$.
\end{lemma}

We are now ready to prove the main result of this section, namely: 

\begin{theorem}\label{thm:diam-k=2}
If $G=(V_0 \cup V_1,E)$ is an absolute retract of bipartite graphs then, with high probability, we can compute $diam(G)$ in $\tilde{\cal O}(m\sqrt{n})$ time. 
\end{theorem}

\begin{proof}
We may assume $diam(G) \geq 3$. Indeed, a bipartite graph has diameter at most two if and only if it is complete bipartite.
First, we compute the peripheral vertices of $H_0$ and $H_1$.
For that, we start computing a $2$-approximation of $diam(G)$ in linear time ({\it e.g.}, by computing the eccentricity of an arbitrary vertex). Let $D$ be the resulting value. There are two cases.
\begin{itemize}
\item If $D < \sqrt{n}$ then, we compute the least $k$ s.t. all vertices of $H_0$ (resp., of $H_1$) have eccentricity at most $k$. For that, it is sufficient to perform a one-sided binary search where, at each step, we apply Lemma~\ref{lem:cst-ecc}. Note that this value $k$ computed is in fact $diam(H_0)$ (resp., $diam(H_1)$). In particular, the total running time is in $\tilde{\cal O}(mk) = \tilde{\cal O}(mD) = \tilde{\cal O}(m\sqrt{n})$. Then, in order to compute the peripheral vertices of $H_0$ (resp., of $H_1$), it is sufficient to apply Lemma~\ref{lem:cst-ecc} one more time in order to compute the vertices of eccentricity at most $k-1$.  
\item Otherwise, $D \geq \sqrt{n}$, and we apply Algorithm~\ref{alg:large-diam} to both $H_0$ and $H_1$. Since by Lemma~\ref{lem:half-helly}, both $H_0$ and $H_1$ are Helly graphs, we have by Lemma~\ref{lem:large-ecc} that the output of Algorithm~\ref{alg:large-diam} is correct with high probability. Furthermore, since any breadth-first search in either $H_0$ or $H_1$ can be simulated with a breadth-first search in $G$, the total running time for executing Algorithm~\ref{alg:large-diam} is with high probability (by Lemma~\ref{lem:large-ecc}) in $\tilde{\cal O}(mn/D) = \tilde{\cal O}(m\sqrt{n})$.    
\end{itemize} 
Finally, in order to compute $diam(G)$ from the peripheral vertices of $H_0$ and $H_1$, we apply the criterion of Lemma~\ref{lem:diam-bipartite}. For that, it is sufficient to scan the neighbourhood of each peripheral vertex of $H_0$ and $H_1$, and therefore it can be done in linear time. 
\end{proof}

\subsection{Chordal bipartite graphs}\label{sec:chordal}

We improve Theorem~\ref{thm:diam-k=2} for the special case of chordal bipartite graphs.
Recall (amongst many characterizations) that a bipartite graph is chordal bipartite if and only if every induced cycle has length four~\cite{GoG78}. It was proved in~\cite{BDS87} that every chordal bipartite graph is an absolute retract of bipartite graphs.

\begin{theorem}\label{thm:chordal}
If $G=(V,E)$ is chordal bipartite then we can compute all the eccentricities (and so, the diameter) in linear time. 
\end{theorem}

The remaining of this section is devoted to a proof of Theorem~\ref{thm:chordal}. For that, we subdivide our proof into four main steps (undefined terminology below is introduced step by step in this section):
\begin{enumerate}
\item We base ourselves on the results from~\cite{BDCV98,DraV96} in order to prove that $H_0$ and $H_1$ -- the two Helly graphs induced by the partite sets of $G$ in its square -- are strongly chordal graphs.
\item The same as in Sec.~\ref{sec:k2-diam}, in general we cannot compute $H_0$ and $H_1$ from $G$ in subquadratic time. In order to overcome this issue, we explain how to compute a clique-tree for these two graphs. 
\item Then, we present an algorithm in order to compute all the eccentricities for strongly chordal graphs, being given as input the clique-tree of such graph ({\it i.e.}, instead of its adjacency list).
\item Finally, we present an ``all eccentricities'' version of Lemma~\ref{lem:diam-bipartite}, and we explain how to solve the corresponding algorithmic problem for chordal bipartite graphs.  
\end{enumerate}

\paragraph{The chordal structure of the partite sets.} A graph is chordal if it has no induced cycle of length more than three. It is strongly chordal if it is chordal and it does not contain any $n$-sun ($n \geq 3$) as an induced subgraph~\cite{FaM83}. A dually chordal graph is a Helly graph in which the intersection graph of balls is chordal~\cite{BaC08,BDCV98}; for other characterizations of this graph class, see~\cite{BDCV98}. The relation between dually chordal graphs and strongly chordal graphs is as follows:

\begin{lemma}[\cite{BDCV98}]\label{lem:strongly-chordal}
A graph is strongly chordal if and only if each induced subgraph is dually chordal.
\end{lemma}

The $k^{th}$-iterated neighbourhood of a vertex $V$, denoted $N^k_G(v)$, is defined recursively as: $N_G^1(v) = N_G(v)$ (open neighbourhood) and $N_G^{k+1}(v) = N_G(N_G^k(v))$~\cite{DraV96}. The following observation was used implicitly in Lemma~\ref{lem:cst-ecc}:

\begin{lemma}\label{lem:iterated-ball}
If $G=(V_0\cup V_1,E)$ is bipartite then, for every $k \geq 1$ and $v \in V_0$, $N_G^{2k}(v) = N_{H_0}^k[v]$.
\end{lemma}

\begin{proof}
By induction, for every $k \geq 1$, $N_G^{2k-1}(v) = N_G^{2k-1}[v] \cap V_1$ and $N_G^{2k}(v) = N_G^{2k}[v] \cap V_0$. The lemma follows since we have $N_{H_0}^k[v] = N_G^{2k}[v] \cap V_0$. 
\end{proof}

For a vertex $v$ in $G$, a maximum neighbour\footnote{This terminology is sometimes used, with a different meaning, for dually chordal graphs~\cite{BDCV98}.} is a vertex $u \in N_G(v)$ such that, for any other $w \in N_G(v)$, we have $N_G(w) \subseteq N_G(u)$. A vertex $v' \neq v$ such that $N_G(v) \subseteq N_G(v')$ is said to cover $v$. Finally a maximum neighbourhood ordering of $G$ is a total ordering $V = (v_1,v_2,\ldots,v_n)$ of its vertex-set such that, for every $1 \leq i \leq n-2$, the vertex $v_i$ has a maximum neighbour and is covered in the induced subgraph $G_i := G \setminus \{v_1,v_2,\ldots,v_{i-1}\}$. We do not use the existence of a maximum neighbourhood ordering directly in our proofs, but rather the following two related results:

\begin{lemma}[\cite{DraV96}]\label{lem:min}
$G=(V,E)$ is chordal bipartite if and only if every induced subgraph of $G$ has a maximum neighbourhood ordering.
\end{lemma}

\begin{lemma}[items (iii) and (v) of the main theorem in~\cite{DraV96}]\label{lem:dragan}
$G$ has a maximum neighbourhood ordering if and only if the system of all iterated neighbourhoods has the Helly property and its intersection graph is chordal.
\end{lemma}

The next result now follows by combining Lemmas~\ref{lem:strongly-chordal}--\ref{lem:dragan}.

\begin{lemma}\label{lem:partite-set-strongly-chordal}
If $G=(V_0\cup V_1,E)$ is chordal bipartite, then $H_0$ and $H_1$ are strongly chordal.
\end{lemma}

\begin{proof}
By symmetry, it suffices to prove the result for $H_0$. For that, by Lemma~\ref{lem:partite-set-strongly-chordal}, it is sufficient to prove that every induced subgraph of $H_0$ is dually chordal. Let $U \subseteq V_0$ be arbitrary. We define $G_U$ as the subgraph induced by $N_G[U]$ (the union of $U$ and of all the vertices of $V_1$ with a neighbour in $U$). By construction, $H_0[U]$ is exactly the graph with vertex-set $U$ and an edge between every two vertices with a common neighbour in $G_U$. Furthermore, since the class of chordal bipartite graphs is hereditary, $G_U$ is chordal bipartite, and so, by~\cite{BDS87}, an absolute retract of bipartite graphs. It thus follows from Lemma~\ref{lem:half-helly} that $H_0[U]$ is a Helly graph. In order to prove that $H_0[U]$ is dually chordal, it now suffices to prove that the intersection graph of its balls is chordal. By the combination of Lemmas~\ref{lem:min} and~\ref{lem:dragan}, the intersection graph $I_U$ of all iterated neighbourhoods of $G_U$ is chordal. By Lemma~\ref{lem:iterated-ball}, the intersection graph of all balls of $H_0[U]$ is an induced subgraph of $I_U$ and therefore, it is also chordal.   
\end{proof}

\paragraph{Computation of a clique-tree.} For a graph $H=(V,E)$, a clique-tree is a tree $T$ whose nodes are the maximal cliques of $H$ and such that, for every $v \in V$, the maximal cliques of $H$ containing $v$ induce a connected subtree $T_v$ of $T$. It is well-known that $H$ is chordal if and only if it has a clique-tree~\cite{Bun74,Gav74,Wal72}. Our intermediate goal is, given a chordal bipartite graph $G = (V_0 \cup V_1, E)$, to compute a clique-tree for $H_0$ and $H_1$ (that are chordal graphs by the above Lemma~\ref{lem:partite-set-strongly-chordal}).

\smallskip
\noindent
A hypergraph ${\cal H} = (X,{\cal R})$ is called a dual hypertree if there exists a tree $T$ whose nodes are the hyperedges in ${\cal R}$ and such that, for every $x \in X$, the hyperedges containing $x$ induce a connected subtree $T_x$ of $T$ (such tree $T$ is called a join tree of ${\cal H}$). Note that dual hypertrees can be recognized in linear time~\cite{TaY84}. Furthermore, we have: 

\begin{lemma}[\cite{DraV96}]\label{lem:hypertree}
If $G=(V_0\cup V_1,E)$ is chordal bipartite then, both hypergraphs $(V_0,\{N_G(v) \mid v \in V_1\})$ and $(V_1,\{N_G(v) \mid v \in V_0\})$ are dual hypertrees.
\end{lemma}

\begin{corollary}\label{cor:clique-tree}
If $G=(V_0\cup V_1,E)$ is chordal bipartite then, we can compute a clique-tree for $H_0$ and $H_1$ in linear time.
\end{corollary}

\begin{proof}
By symmetry, we only prove the result for $H_0$. Consider the hypergraph ${\cal H} = (V_0,\{N_G(v) \mid v \in V_1\})$. Note that the underlying graph of ${\cal H}$ (obtained by adding an edge between every two vertices that are contained in a common hyperedge of ${\cal H}$) is exactly $H_0$. Since by Lemma~\ref{lem:hypertree}, ${\cal H}$ is a dual hypertree, every maximal clique of the underlying graph $H_0$ must be a hyperedge of ${\cal H}$~\cite{BFMY83}. Then, let us reduce ${\cal H}$, {\it i.e.}, we remove all hyperedges that are strictly contained into another hyperedge. Again since ${\cal H}$ is a dual hypertree, the resulting reduced hypergraph ${\cal H}'$ can be computed in linear time~\cite{TaY84}. Furthermore, by construction, the hyperedges of ${\cal H}'$ are exactly the maximal cliques of $H_0$. Let us construct a join tree of ${\cal H}'$. It can be done in linear time~\cite{TaY84}. We are done as such join tree is a clique-tree of $H_0$. 
\end{proof}

\paragraph{Computation of all the eccentricities in the partite sets.} Next, we propose a new algorithm in order to compute all the eccentricities of a strongly chordal graph $H$, being given a clique-tree. We often use in our proof the clique-vertex incidence graph of $H$, {\it i.e.}, the bipartite graph whose partite sets are the vertices and the maximal cliques of $H$, and such that there is an edge between every vertex of $H$ and every maximal clique of $H$ containing it. 

Let us first recall the following result about the eccentricity function of Helly graphs:

\begin{lemma}[\cite{Dra89}]\label{lem:unimodal}
If $H=(V,E)$ is Helly then, for every vertex $v$ we have $e_H(v) = d_H(v,C(H)) + rad(H)$, where $C(H)$ denotes the set of central vertices of $H$. 
\end{lemma}

Hence, by Lemma~\ref{lem:unimodal}, we are left computing $C(H)$. It starts with computing one central vertex. Define, for every vertex $v$ and vertex-subset $C$, $d_H(v,C) = \min_{c \in C} d_H(v,c)$. Following~\cite{ChD94}, we call a set $C$ gated if, for every $v \notin C$, there exists a vertex $v^* \in N_H^{d_H(v,C)-1}[v] \cap \left( \bigcap \{ N_H(c) \mid c \in C, \ d_H(v,c) = d_H(v,C) \} \right)$ (such vertex $v^*$ is called a gate of $v$).

\begin{lemma}[\cite{ChN84}]\label{lem:gated-clique}
Every clique in a chordal graph is a gated set.
\end{lemma}

\begin{lemma}[\cite{DuD19+}]\label{lem:compute-gate}
If $T$ is a clique-tree of a chordal graph $H$ then, for every clique $C$ of $H$, for every $v \notin C$ we can compute $d_H(v,C)$ and a corresponding gate $v^*$ in total ${\cal O}(w(T))$ time, where $w(T)$ denotes the sum of cardinalities of all the maximal cliques of $H$. 
\end{lemma}

For every $u,v \in V$ and $k \leq d_H(u,v)$, the set $L_H(u,k,v) = \{ x \in I_H(u,v) \mid d_H(u,x) = k\}$ is called a slice. We also need the following result about slices in chordal graphs:

\begin{lemma}[\cite{ChN84}]\label{lem:slice}
Every slice in a chordal graph is a clique.
\end{lemma}

Now, consider the procedure described in Algorithm~\ref{alg:central} in order to compute a central vertex.

\begin{algorithm}
	\caption{Computation of a central vertex.}
	\label{alg:central}
	\footnotesize
	
	\begin{algorithmic}[1]
		\REQUIRE{A strongly chordal graph $H$.}\\
		\STATE{$v \leftarrow$ an arbitrary vertex of $H$}
		\STATE{$u \leftarrow$ a furthest vertex from $v$, {\it i.e.}, $d_H(u,v) = e_H(v)$}
		\STATE{$w \leftarrow$ a furthest vertex from $u$, {\it i.e.}, $d_H(u,w) = e_H(u)$}
		\FORALL{$r \in \{ \left\lceil e_H(u)/2 \right\rceil, \left\lceil (e_H(u)+1)/2 \right\rceil, 1 + \left\lceil e_H(u)/2 \right\rceil \}$}
		\STATE{Set $C := L(w,r,u)$ {\it //$C$ is a clique by Lemma~\ref{lem:slice}}}
		\FORALL{$v \notin C$}
		\STATE{Compute $d_H(v,C)$ and a corresponding gate $v^*$ {\it //whose existence follows from Lemma~\ref{lem:gated-clique}}}
		\ENDFOR
		\STATE{Set $S := \{ v^* \mid d_H(v,C) = r\}$ {\it //gates of vertices at max. distance from $C$}}
		\FORALL{$c \in C$}
		\IF{$S \subseteq N_H(c)$}
		\RETURN{$c$}
		\ENDIF
		\ENDFOR
		\ENDFOR
	\end{algorithmic}
\end{algorithm}

\begin{lemma}[special case of Theorem 5 in~\cite{DuD19+}]\label{lem:alg-central}
Algorithm~\ref{alg:central} outputs a central vertex of $H$.
\end{lemma}

\begin{lemma}\label{lem:impl-alg-central}
If $T$ is a clique-tree of a strongly chordal graph $H$ then, we can implement Algorithm~\ref{alg:central} in order to run in ${\cal O}(w(T))$ time, where $w(T)$ denotes the sum of cardinalities of all the maximal cliques of $H$. 
\end{lemma}

\begin{proof}
Lines 1--3 require executing breadth-first searches in $H$. It can be done in ${\cal O}(w(T))$ time by executing breadth-first searches in the clique-vertex incidence graph $I_H$ of $H$. Note that we can compute $I_H$ in ${\cal O}(w(T))$ time from the clique-tree $T$. Now, consider any of the at most three executions of the for loop starting at Line 4. The subset $C$ at Line 5 can also be computed in ${\cal O}(w(T))$ time using breadth-first searches in $I_H$. Then, for implementing the computation of all the gates, at Line 6--7, we call Lemma~\ref{lem:compute-gate}. We are left explaining how to implement the internal for loop, starting at Line 9, so that it runs in total ${\cal O}(w(T))$ time. For that, we assign a counter $g(c)$ for every $c \in C$ (initially equal to $0$), whose final value must be $|N_H(c) \cap S|$. Doing so, the test at Line 10 boilds down to verifying whether $g(c) = |S|$. In order to correctly set the values of the counters $g(c), \ c \in C$, we root the clique-tree $T$ arbitrarily and then we perform a breadth-first search of this tree $T$ starting from the root. For each maximal clique $K$ of $H$, let $K'$ be the common intersection with its father node in $T$ (in particular, $K' = \emptyset$ if $K$ is the root). For every $c \in C \cap (K \setminus K')$, we increase the counter $g(c)$ by exactly $|K \cap S|$. However, for every $c \in C \cap K'$, we only increase the counter $g(c)$ by $|S \cap (K \setminus K')|$; indeed, for these vertices, the contribution of the gates in $S \cap K'$ was already counted earlier during the breadth-first search. Overall, we just need to scan each maximal clique of $H$ once, and so, the running time is in ${\cal O}(w(T))$, as desired.  
\end{proof}

Then, given a central vertex $c$ of $H$, we explain how to compute $C(H)$ by local search in the neighbourhood at distance two around $c$. For that, we need one more structural result about the center of strongly chordal graphs, namely:

\begin{lemma}[\cite{Dra89,Dra93}]\label{lem:center-strongly-chordal}
If $H$ is strongly chordal then, its center $C(H)$ induces a strongly chordal graph of radius $\leq 1$. 
\end{lemma}

We also need the following nice characterization of strongly chordal graphs, namely:

\begin{lemma}[\cite{Bra91,FaM83}]\label{lem:incidence}
$H$ is strongly chordal if and only if its clique-vertex incidence graph $I_H$ is chordal bipartite.
\end{lemma}

\begin{proposition}\label{prop:center}
If $T$ is a clique-tree of a strongly chordal graph $H = (V,E)$ then, we can compute its center $C(H)$ in ${\cal O}(w(T))$ time. 
\end{proposition}

\begin{proof}
Let $c \in C(H)$ be a fixed central vertex of $H$. By the combination of Lemmas~\ref{lem:alg-central} and~\ref{lem:impl-alg-central}, it can be computed in ${\cal O}(w(T))$ time. We also compute the clique-vertex incidence graph $I_H$ from $T$, in ${\cal O}(w(T))$ time. Let $r := rad(H)$ (computable in ${\cal O}(w(T))$ time by using a breadth-first search rooted at $c$ in $I_H$). 
Note that if $r \leq 2$, then in order to compute $C(H)$, we are left computing all the vertices of eccentricity at most $r$ in $H$. For that, since by Lemma~\ref{lem:incidence} $I_H$ is chordal bipartite, we can use Lemma~\ref{lem:cst-ecc}. It takes time linear in the size of $I_H$, and so it takes ${\cal O}(w(T))$ time.
From now on, we assume $r \geq 3$. 
By Lemma~\ref{lem:center-strongly-chordal}, $C(H) \subseteq N_H^2[c]$. In particular,  If $v \in V$ is such that $d_H(v,c) \leq r-2$ then it is trivially at a distance $\leq r$ from every vertex of $N_H^2[c]$. Thus, in order to compute $C(H)$ from $N_H^2[c]$, we only need to consider the subsets $A_{r-1} := \{ v \in V \mid d_H(v,c) = r-1\}$ and $A_r := \{ v \in V \mid d_H(v,c) = r \}$. We prove as an intermediate claim that $N_H[c]$ is gated. Indeed, for every $v \notin N_H[c]$, the closest vertices to $v$ in this closed neighbourhood are those in the slice $L(v,d_H(v,c)-1,c)$, and thus they induce a clique by Lemma~\ref{lem:slice}. Then, the claim follows from Lemma~\ref{lem:gated-clique}. Furthermore, for every $v \notin N_H[c]$, we can compute a corresponding gate $v^*$ as follows: we contract in $T$ the subtree $T_c$ (induced by all the maximal cliques containing $c$) to a single node representing $N_H[c]$, then we apply Lemma~\ref{lem:compute-gate}. It takes ${\cal O}(w(T))$ time.

Now, let $v \in A_{r-1} \cup A_r$. If the closed neighbourhood of any vertex $u$ intersects $N_H[c] \cap N_H^{r-1}[v]$ then, $d_H(u,v) \leq r$. Conversely, we claim that the closed neighbourhood of any central vertex must intersect $N_H[c] \cap N_H^{r-1}[v]$. Indeed, observe that for every $c' \in C(H)$, $N_H[c'] \cap N_H^{r-1}[v] \neq \emptyset$. By Lemma~\ref{lem:center-strongly-chordal}, the balls $N_H[c'], \ c' \in C(H)$ also pairwise intersect. Therefore, by the Helly property, there exists a $x \in N^{r-1}_H[v] \cap \left( \bigcap_{c' \in C(H)} N_H[c'] \right)$. We are done as in this situation $x \in N_H[c] \cap N_H^{r-1}[v]$.
We now analyse the following two cases:

\smallskip
{\bf Case $v \in A_r$.} Let $v^*$ be the corresponding gate. Note that $d_H(v^*,c) = d_H(v,N_H[c]) - 1 = d_H(v,c) - 2 = r-2$. In particular, every vertex of $N_H^2[v^*]$ is at a distance $\leq r$ from vertex $v$. Conversely, the closed neighbourhood of every central vertex must intersect $N_H[c] \cap N^{r-1}_H[v] = L(v,r-1,c) \subseteq N_H[v^*]$. As a result, $C(H) \subseteq N_H^2[v^*]$. 
%
%\smallskip
%Then, let $B := \{c\} \cup \{ v^* \mid d_H(v,c) = r \}$. We want to compute $U_r := \{ x \in V \mid \forall b \in B, d_H(x,b) \leq 2 \}$ which is by the above a superset of $C(H)$. For that, since $I_H$ is chordal bipartite (Lemma~\ref{lem:incidence}) we can adapt the technique of Lemma~\ref{lem:cst-ecc}. It takes time linear in the size of $I_H$, and so, it can be done in ${\cal O}(w(T))$ time.

\smallskip
{\bf Case $v \in A_{r-1}$.} Let $Z = \{ z \in N_H[c] \mid d_H(v,c) = r-1 \}$. Note that $c \in Z$. Since the balls $N^{r-2}_H[v]$ and $N_H[z], \ z \in Z$ pairwise intersect, by the Helly property there exists a vertex $v' \in N_H^{r-2}[v] \cap \left( \bigcap \{ N_H(z) \mid z \in Z \} \right)$. Observe that $v' \in N_H^{r-2}[v] \cap N_H[c] = L(v,r-2,c)$. In particular, every vertex of $N_H^2[v']$ is at a distance $\leq r$ from $v$. Conversely, since by Lemma~\ref{lem:slice} the set $L(v,r-2,c)$ is a clique, $N_H^{r-1}[v] \cap N_H[c] = L(v,r-2,c) \cup Z \subseteq N_H[v']$. Thus, $C(H) \subseteq N_H^2[v']$.

\smallskip
Note that we may choose as our $v'$ any $y \in N_H(v^*)$ that maximizes $|N_H[y] \cap N_H[c]|$, where $v^*$ is the gate computed for $v$. In order to compute such vertex $v'$, for every $v \in A_{r-1}$, we use a similar trick as for Lemma~\ref{lem:impl-alg-central}. Namely, for every $y \in V$,  we assign a counter $h(y)$ (initially equal to $0$), whose final value must be $|N_H[y] \cap N_H[c]|$. In order to correctly set the values of these counters, we root the clique-tree $T$ arbitrarily and then we perform a breadth-first search of this tree $T$ starting from the root. For each maximal clique $K$ of $H$, let $K'$ be the common intersection with its father node in $T$ (in particular, $K' = \emptyset$ if $K$ is the root). For every $y \in K \setminus K'$, we increase the counter $h(y)$ by exactly $|K \cap N_H[c]|$. However, for every $y \in K'$, we only increase the counter $h(y)$ by $|N_H[c] \cap (K \setminus K')|$. Overall, we just need to scan each maximal clique of $H$ once, and so, the running time is in ${\cal O}(w(T))$.  Finally, we scan each maximal clique $K$ once more and, to every vertex of $K$, we assign a vertex $y \in K$ maximizing $h(y)$. For every $v \in A_{r-1}$, we choose for $v'$ any vertex $y$ assigned to $v^*$ and maximizing $h(y)$. 

\smallskip
Overall, let $B := \{c\} \cup \{ v^* \mid v \in A_r \} \cup \{ v' \mid v \in A_{r-1} \}$. By the above case analysis, $C(H) = \{ x \in V \mid \forall b \in B, \ d_H(b,x) \leq 2 \}$. Since $I_H$ is chordal bipartite (Lemma~\ref{lem:incidence}) we can adapt the technique of Lemma~\ref{lem:cst-ecc} in order to compute $C(H)$ ({\it i.e.}, we set $k=2$, and then compute partitions of the subset $B$ instead of computing partitions for the full half-set $V$). It takes time linear in the size of $I_H$, and so, it can be done in ${\cal O}(w(T))$ time.
\end{proof}

\paragraph{Computation of all the eccentricities in $G$.} Before proving Theorem~\ref{thm:chordal}, we need a final ingredient. Let us first generalize Lemma~\ref{lem:diam-bipartite} as follows.

\begin{lemma}\label{lem:ecc}
If $G=(V_0 \cup V_1,E)$ is an absolute retract of bipartite graphs then, the following holds for every $i \in \{0,1\}$ and $v \in V_i$:
\begin{itemize}
\item If $e_{H_i}(v) \leq rad(H_{1-i}) - 1$ then, $e_G(v) = 2e_{H_i}(v) + 1 = 2rad(H_{1-i}) - 1$.

\item If $e_{H_i}(v) = rad(H_{1-i})$ then, $e_G(v) = 2rad(H_{1-i})$ if and only if $N_G(v) \subseteq C(H_{1-i})$ and, for every $u \in V_{1-i}$, we have $d_{H_{1-i}}(u,N_G(v)) \leq rad(H_{1-i}) - 1$ (otherwise, $e_G(v) = 2rad(H_{1-i}) + 1$).

\item If $e_{H_i}(v) \geq rad(H_{1-i}) + 1$ then, $e_G(v) = 2 e_{H_i}(v)$ if and only if we have $e_{H_{1-i}}(u) < e_{H_i}(v)$ for some neighbour $u \in N_G(v)$ (otherwise, $e_G(v) = 2e_{H_i}(v)+1$).
\end{itemize}
\end{lemma}

\begin{proof}
Let us first consider the case $e_{H_i}(v) \leq rad(H_{1-i}) - 1$. In particular, every neighbour $u \in N_G(v)$ is at a distance $\leq 1 + 2e_{H_i}(v) + 1 \leq 2rad(H_{1-i})$ from any vertex of $V_{1-i}$. Therefore, $e_{H_i}(v) = rad(H_{1-i}) - 1$, and every vertex of $N_G(v)$ is central in $H_{1-i}$. Observe that $v$ must be at a distance $\geq 2 rad(H_{1-i})-1$ from at least one vertex of $V_{1-i}$ since otherwise, the eccentricity of all the vertices of $N_G(v)$, in $H_{1-i}$, would be $ < rad(H_{1-i})$. As a result, $e_G(v) = 2 rad(H_{1-i})-1$.

Now, consider the case $e_{H_i}(v) = rad(H_{1-i})$. Then, since $G$ is bipartite, we have $e_G(v) = 2rad(H_{1-i})$ if and only if every vertex of $V_{1-i}$ is at a distance $\leq 2rad(H_{1-i})-1$ from vertex $v$ (otherwise, $e_G(v) = 2rad(H_{1-i})+1$). Equivalently, $e_G(v) = 2rad(H_{1-i})$ if and only if, for every $u \in V_{1-i}$, we have $d_G(u,N_G(v)) \leq 2rad(H_{1-i})-2$. Again, since $d_{H_{1-i}}(u,N_G(v)) = d_G(u,N_G(v))/2$, the latter is equivalent to have $d_{H_{1-i}}(u,N_G(v)) \leq rad(H_{1-i}) - 1$. Note that since $N_G(v)$ is a clique of $H_{1-i}$, this last equivalence also implies that $N_G(v) \subseteq C(H_{1-i})$. 

Finally, consider the case $e_{H_i}(v) \geq rad(H_{1-i}) + 1$.
%If $e_{H_i}(v) = 0$ then, $V_i = \{v\}$ and $G$ must be a star. In particular, $e_G(v) = 1$ and $N_G(v) = V_{1-i}$. 
%Otherwise, observe that for every $v,v' \in V_i$ we have $d_G(v,v') = 2d_{H_i}(v,v')$. This implies $2e_{H_i}(v) \leq e_G(v) \leq  2e_{H_i}(v) + 1$. In the special case $e_{H_i}(v) = 1$ we get $e_G(v) \in \{2,3\}$. Since $G$ is bipartite, $e_G(v) = 2$ if and only if $N_G(v) = V_{1-i}$. In the remaining of the proof, we will assume $e_{H_i}(v) \geq 2$. 
If furthermore there is a neighbour $u \in N_G(v)$ s.t. $e_{H_{1-i}}(u) < e_{H_i}(v)$ then, vertex $v$ is at a distance at most $2e_{H_{1-i}}(u) + 1 < 2e_{H_i}(v)$ from every vertex of $V_{1-i}$, and so, $e_G(v) = 2e_{H_i}(v)$. Now, let us assume to have $e_{H_{1-i}}(u) \geq e_{H_i}(v)$ for every neighbour $u \in N_G(v)$. Suppose for the sake of contradiction $e_G(v) < 2e_{H_i}(v) + 1$. In particular, every vertex of $V_{1-i}$ must be at a distance at most $2e_{H_i}(v) - 1$ from vertex $v$. Equivalently for any $x \in V_{1-i}$, the half-balls $N_G(v)$ and $N_G^{2e_{H_i}(v)-2}[x] \cap V_{1-i}$ intersect. Note that $2e_{H_i}(v)-2 \geq 2rad(H_{1-i}) \geq diam(H_{1-i})$. Hence, the half-balls $N_G^{2e_{H_i}(v)-2}[x] \cap V_{1-i}, \ x \in V_{1-i}$ also pairwise intersect. But then, by Theorem~\ref{thm:bipartite-retract}, there is a neighbour $u \in N_G(v)$ such that $e_{H_{1-i}}(u) \leq (2e_{H_i}(v)-2)/2 = e_{H_i}(v) - 1$. A contradiction.   
\end{proof}

Of the three cases in the above Lemma~\ref{lem:ecc}, the real algorithmic challenge is the case $e_{H_i}(v) = rad(H_{1-i})$, for some $i \in \{0,1\}$. Finally, we prove that such case can be solved in linear time for chordal bipartite graphs. 

\begin{proof}[Proof of Theorem~\ref{thm:chordal}]
By Lemma~\ref{lem:partite-set-strongly-chordal}, $H_0$ and $H_1$ are strongly chordal graphs. We compute clique-trees $T_0$ and $T_1$ for both $H_0$ and $H_1$, that takes linear time by Corollary~\ref{cor:clique-tree}. Note that in particular, for each $i \in \{0,1\}$ we get $w(T_i) = {\cal O}(m+n)$, with $w(T_i)$ the sum of the cardinalities of the maximal cliques of $H_i$. Then, we compute all the eccentricities of $H_0$ (resp, of $H_1$), that is ${\cal O}(w(T_0))$-time equivalent to computing $C(H_0)$ according to Lemma~\ref{lem:unimodal} (resp., ${\cal O}(w(T_1))$-time equivalent to computing $C(H_1)$). By Proposition~\ref{prop:center}, the center can be computed in ${\cal O}(w(T_0))$ time (resp., in ${\cal O}(w(T_1))$ time). Overall, for each $i \in \{0,1\}$ and every $v \in V_i$, we so computed $e_{H_i}(v)$. It takes total linear time. Doing so, we can also compute $rad(H_0)$ and $rad(H_1)$ within the same amount of time. Finally, we are left explaining how to deduce from the latter the eccentricities $e_G(v), \ v \in V_0$ (the case $v \in V_1$ is symmetric to this one). By Lemma~\ref{lem:ecc}, if $e_{H_0}(v) \neq rad(H_1)$ then, we can compute $e_G(v)$ by scanning the neighbour set $N_G(v)$. Therefore, in total ${\cal O}(m)$ time, we can compute $e_G(v)$ for every $v \in V_0$ s.t. $e_{H_0}(v) \neq rad(H_1)$. In the same way, for every $v \in V_0$ s.t. $e_{H_0}(v) = rad(H_1)$ and $N_G(v) \not\subseteq C(H_1)$, we set directly $e_G(v) = 2rad(H_1)+1$, that is correct by Lemma~\ref{lem:ecc}, and it also takes ${\cal O}(m)$ time in total.

Let $S := \{ v \in V_0 \mid e_{H_0}(v) = rad(H_1), \ N_G(v) \subseteq C(H_1) \}$. By Lemma~\ref{lem:ecc}, every vertex of $S$ has eccentricity either $2rad(H_1)$ or $2rad(H_1)+1$ in $G$. We may further assume $rad(H_1) \geq 3$ because otherwise, we can compute the eccentricity of all the vertices of $S$ in total ${\cal O}(m)$ time by applying the technique of Lemma~\ref{lem:cst-ecc}. In this situation, we define a set $W \subseteq V_1$ with the following property: for every $v \in S$, we have $e_G(v) = 2rad(H_1)$ if and only if for every $w \in W$ we have $d_G(v,w) \leq 3$. Note that doing so, we may compute the eccentricity of all the vertices of $S$ in ${\cal O}(m)$ time by applying the same technique as for Lemma~\ref{lem:cst-ecc} ({\it i.e.}, we set $k=3$, then we compute partitions for $W$ rather than for the full half-set $V_1$). Furthermore, in order to compute this set $W$, we proceed in a quite similar fashion as for Proposition~\ref{prop:center}. We detail this procedure next.

{\bf The algorithm.}
Let $r = rad(H_1)$.
First, we compute a vertex $c \in C(H_1)$ that is adjacent to all the central vertices of $H_1$ (such vertex is guaranteed to exist by Lemma~\ref{lem:center-strongly-chordal}). For every $u \in V_1$ s.t. $d_{H_1}(u,c) \geq r-1$, we compute a vertex $u^* \in N_{H_1}^{r-2}[u] \cap \left( \bigcap\{ N_{H_1}[x] \mid x \in N_{H_1}[c] \cap N^{r-1}_{H_1}[u] \} \right)$ and we add this gate $u^*$ to the set $W$.

{\bf Correctness.} By Lemma~\ref{lem:ecc}, for every $v \in S$, we have $e_G(v) = 2rad(H_1)$ if and only if, for every $u \in V_1, \ d_{H_1}(u,N_G(v)) \leq r-1$. Note that if $d_{H_1}(u,c) \leq r-2$ then, the distance in $H_1$ between $u$ and any vertex of $C(H_1)$ is at most $r-1$. So, we are only interested in those $u$ s.t. $d_{H_1}(u,c) \geq r-1$. Assume the existence for such vertex $u$ of a $u^* \in N_{H_1}^{r-2}[u] \cap \left( \bigcap\{ N_{H_1}[x] \mid x \in N_{H_1}[c] \cap N^{r-1}_{H_1}[u] \} \right)$. If $d_G(v,u^*) \leq 3$ then, $d_G(u,v) \leq 3+2(r-2) = 2r-1$, as desired. Conversely, if $e_G(v) = 2rad(H_1)$ then there must be a $c_u \in N_G(v) \subseteq C(H_1)$ s.t. $d_{H_1}(c_u,u) \leq r-1$. In particular, $c_u \in N_{H_1}[c] \cap N^{r-1}_{H_1}[u] \subseteq N_{H_1}[u^*]$ and therefore, $d_G(v,u^*) \leq d_G(v,c_u) + d_G(c_u,u^*) = 1 + 2d_{H_1}(c_u,u^*) \leq 3$. The existence of a $u^*$ as above was already proved in Proposition~\ref{prop:center}, but we repeat here the arguments for completeness. There are two cases. If $d_{H_1}(u,c) = r$ then, $N_{H_1}[c] \cap N^{r-1}_{H_1}[u] = L_{H_1}(u,r-1,c)$, and the result follows from the combination of Lemmas~\ref{lem:slice} and~\ref{lem:gated-clique}. Otherwise, $d_{H_1}(u,c) = r-1$, and let $Z = \{ z \in N_{H_1}[c] \mid d_{H_1}(u,z) = r-1 \}$. Since the balls $N_{H_1}^{r-2}[u]$ and $N_{H_1}[z], \ z \in Z$ pairwise intersect, by the Helly property, there exists a $u^* \in N_{H_1}^{r-2}[u] \cap \left( \bigcap \{ N[z] \mid z \in Z \} \right)$. Observe that by construction, $u^* \in L_{H_1}(u,r-2,c)$, that is a clique according to Lemma~\ref{lem:slice}. We are done as $N^{r-1}_{H_1}[u] \cap N_{H_1}[c] = L_{H_1}(u,r-2,c) \cup Z \subseteq N_{H_1}[u^*]$.

{\bf Complexity.} The constructive proof of Proposition~\ref{prop:center} yields an ${\cal O}(w(T_1))$-time algorithm in order to compute $W$. Therefore, this set $W$ can be constructed in linear time. 
\end{proof}

\section{$k$-chromatic graphs}\label{sec:k-chromatic}

Recall that a proper $k$-coloring of $G=(V,E)$ is any mapping $c : V \to \{1,2,\ldots,k\}$ such that $c(u) \neq c(v)$ for every edge $uv \in E$. The chromatic number of $G$ is the least $k$ such that it has a proper $k$-coloring, and a $k$-chromatic graph is a graph whose chromatic number is equal to $k$. Note in particular with this definition that a $(k-1)$-chromatic graph is not $k$-chromatic. We study the diameter within the absolute retracts of $k$-chromatic graphs, for every fixed $k \geq 3$. 

Our approach requires such graphs to be equipped with a proper $k$-coloring. While this is a classic NP-hard problem for every $k \geq 3$~\cite{JoG79}, it is known that it can be done in polynomial time for absolute retracts of $k$-chromatic graphs~\cite{BaP99}. We remind this colouring algorithm in Sec.~\ref{sec:col} where we observe it can be implemented in order to run in linear time. Our general framework for diameter computation (somehow mimicking what we did in Sec.~\ref{sec:k2-diam}) is presented in Sec.~\ref{sec:framework}. We complete our approach in Sec.~\ref{sec:di}, before concluding this section with its main result in Sec~\ref{sec:main-k}.
%in Sec.~\ref{sec:k-gq-4}, for $k \geq 4$, and in Sec.~\ref{sec:k=3}, for $k=3$, for which additional difficulties arise. 

\subsection{Colouring algorithm}\label{sec:col}

We start with a reminder of the Colouring algorithm presented in~\cite{BaP99} (Algorithm~\ref{alg:col}). The description below is exactly the same as in~\cite{BaP99}, where an ${\cal O}(n^2)$ running-time was claimed.   

\begin{algorithm}
	\caption{Colouring algorithm.}
	\label{alg:col}
	\footnotesize
	
	\begin{algorithmic}[1]
		\REQUIRE{A graph $G=(V,E)$.}\\
		\STATE{Pick a vertex $u$ and let $K$ be a maximal clique containing $u$.}
		\STATE{Let $k := |K|$ and let $c : K \to \{1,2,\ldots,k\}$ be a proper $k$-coloring. For each $v \in K \setminus \{u\}$, colour the common neighbours of $K \setminus \{v\}$ with $c(v)$. Then, let $L$ be the set of vertices coloured so far. For each neighbour $v$ of $u$ not in $L$, there is a unique colour $i$ such that $u$ does not have a neighbour in $L$ with colour $i$ (otherwise, $G$ is not an absolute retract). Assign $c(u) : = i$. If $N_G[u]$ includes all of $G$, then {\bf goto} Step 5.}
		\STATE{For each vertex $v \in V$ s.t. $d_G(u,v) = 2$, there is a unique colour $i$ not occurring in $N_G(u) \cap N_G(v)$ or in $\{u\} \cup (N_G(u) \cap N_G(v))$ (otherwise, $G$ is not an absolute retract). Assign $c(v) := i$. If $N_G^2[u]$ includes all of $G$, then {\bf goto} Step 5, otherwise assign $\ell:=3$.}
		\STATE{For each vertex $v \in V$ s.t. $d_G(u,v) = \ell$, there is a unique colour $i$ not occurring in $N_G^{\ell-1}[u] \cap N_G(v)$ (otherwise, $G$ is not an absolute retract). Assign $c(v) := i$. If $N_G^{\ell}[u]$ does not yet contain all of $G$, then assign $\ell:=\ell+1$ and start Step 4 again.}
		\STATE{If $c(u) = c(v)$ for some edge $uv$, then $G$ is not an absolute retract. Otherwise, $c$ is a proper $k$-coloring of $G$, and $G$ is $k$-chromatic.}
	\end{algorithmic}
\end{algorithm}

We refer to~\cite{BaP99} for a correctness proof of Algorithm~\ref{alg:col}. Our modest, but important contribution for our claimed running-times is as follows:

\begin{proposition}\label{prop:col}
There is a linear-time algorithm such that, for every $k \geq 3$, if the input $G$ is an absolute retract of $k$-chromatic graphs, then it computes a proper $k$-coloring of $G$. 
\end{proposition}

\begin{proof}
Consider the following modified version of Algorithm~\ref{alg:col}:
\begin{enumerate}
\item We start from an arbitrary vertex $u$ and we greedily compute a maximal clique $K$ containing vertex $u$. It takes linear time. Furthermore, let $k := |K|$ and let $c : K \to \{1,2,\ldots,k\}$ be a proper $k$-coloring ({\it i.e.}, obtained by numbering the vertices of $K$ from $1$ to $k$). 
\item For every vertex $x \in V$, we compute $|N_G(x) \cap K|$. It can be done in linear time by scanning once the neighbourhood of each vertex $v \in K$. We consider the vertices $v \in N_G(u) \setminus K$ by non-increasing value of $|N_G(x) \cap K|$. Note that such ordering of $N_G(u) \setminus K$ can be computed using a linear-time sorting algorithm. As in the standard greedy coloring algorithm, we assign to $v$ the least color $i$ not present in its neighbourhood. If $i \geq k+1$ then, $G$ is not an absolute retract and we stop. Doing so, it takes ${\cal O}(|N_G(v)|)$ time in order to color $v$, and so, this whole step takes ${\cal O}(m)$ time.
\item We perform a breadth-first search rooted at $u$. Then, we consider the vertices $v \in V$ s.t. $d_G(u,v) = 2$. First, we search for the least color $i$ such that $i \neq c(u)$ and there is no neighbour of $v$ coloured $i$. If $i \leq k$ then, we set $c(v) := i$. Otherwise, we assign to $v$ the least color $i$ not present in $N_G(v)$ (possibly, $i = c(u)$). If $i \geq k+1$ then, $G$ is not an absolute retract and we stop. This whole step also takes ${\cal O}(m)$ time.
\item Finally, we consider all the remaining vertices $v, \ d_G(u,v) \geq 3$, by non-decreasing value of $d_G(u,v)$ (these distances were computed during the breadth-first search). Here also, we can compute such ordering of the remaining vertices by using a linear-time sorting algorithm. We apply the classic greedy coloring procedure, assigning to the current vertex $v$ the least colour $i$ not present in its neighbourhood. If $i \geq k+1$ then $G$ is not an absolute retract, and we stop. Overall, the whole algorithm indeed runs in linear time.  
\end{enumerate}
In order to prove correctness of this above algorithm, it suffices to prove that if $G$ is an absolute retract, then it computes the same coloring as Algorithm~\ref{alg:col}. This is clear at Step 1, where we only colour the vertices of the maximal clique $K$. Then, at Step 2, we start coloring the neighbours of $u$ with exactly $k-1$ neighbours in $K$. Note that the only possible color amongst $\{1,\ldots,k\}$ to assign to such vertex $x$ is $c(v)$, where $v$ is the unique non-neighbour of $x$ in $K$. This is also what Algorithm~\ref{alg:col} does. Let $L$ be the set of vertices coloured so far. We continue coloring the remaining neighbours $v \in N_G(u) \setminus L$. Algorithm~\ref{alg:col} exploits the property that such vertices have exactly one color $i \in \{1,\ldots,k\}$ that is not present amongst their neighbours in $L$. But then, the classic greedy coloring procedure also assigns this color $i$ to $v$. Next, consider a vertex $v$ s.t. $d_G(u,v) = 2$. If there is a unique available color $i$ amongst $\{1,\ldots,k\}$ that is not present in $N_G(u) \cap N_G(v)$ then, Algorithm~\ref{alg:col} assigns this color to $v$; so does the classic greedy coloring procedure. Otherwise, Algorithm~\ref{alg:col} assigns to $v$ the unique color amongst $\{1,\ldots,k\}$ that is not present in $\{u\} \cup (N_G(u) \cap N_G(v))$; so does our modified greedy coloring procedure, where we first exclude color $c(u)$ from the range of possibilities. Finally, similar arguments apply to the vertices $v$ s.t. $d_G(u,v) \geq 3$.   
\end{proof}

In the remainder of the section, we always assume the input graph $G$ to be given with a proper $k$-coloring. We sometimes use implicitly the fact that, for an absolute retract, such proper $k$-coloring is unique up to permuting the colour classes~\cite{PeP85}. 

\subsection{General framework}\label{sec:framework}

The present section aims at introducing the necessary results and terminology for Sec~\ref{sec:di}.
%Sec.~\ref{sec:k-gq-4} and~\ref{sec:k=3}. 
Recall that in a graph $G=(V,E)$, a vertex $v$ is covered by another vertex $w$ if $N_G(v) \subseteq N_G(w)$ (a covered vertex is called embeddable in~\cite{PeP85}). We now introduce a first characterization of absolute retracts:

\begin{theorem}[\cite{PeP85}]\label{thm:pesch-characterization}
Let $k \geq 3$. The graph $G=(V,E)$ is an absolute retract of $k$-chromatic graphs if and only if for any proper $k$-coloring $c$, every peripheral vertex $v$ is adjacent to all vertices $u$ with $c(u) \neq c(v)$, or it is covered and $G \setminus v$ is an absolute retract of $k$-chromatic graphs.  
\end{theorem}

We highlight the following special case of this characterization, that allows us to deal with the diameter-two case separately from the general case, namely:

\begin{proposition}[\cite{PeP85}]\label{prop:diam2}
Let $G=(V,E)$ be an absolute retract of $k$-chromatic graphs for some $k \geq 3$, and let $c$ be a corresponding proper $k$-coloring. Then, $diam(G) \leq 2$ if and only if for every $1 \leq i \leq k$, there exists a $z_i \in V$ with $z_iv \in E$ for all $v \in V$ s.t. $c(v) \neq i$.
\end{proposition}

\begin{corollary}\label{cor:diam-2}
If $G=(V,E)$ is an absolute retract of $k$-chromatic graphs for some $k \geq 3$, then we can decide whether $diam(G) \leq 2$ in linear time.
\end{corollary}

\begin{proof}
By Proposition~\ref{prop:col}, a proper $k$-coloring $c : V \to \{1,\ldots,k\}$ can be computed in linear time. Then, we apply the criterion of Proposition~\ref{prop:diam2}. For that, for each colour $i$ we choose as our candidate vertex for $z_i$ a maximum-degree vertex of colour $i$.
\end{proof}

%Note that we can check the condition of Proposition~\ref{prop:diam2} in linear time, as follows. For each colour $i$ we choose as our candidate vertex for $z_i$ a maximum-degree vertex of colour $i$. 
Thus, in what follows, we focus on the case $diam(G) \geq 3$. For that, we use in combination to Theorem~\ref{thm:pesch-characterization} yet another characterization of absolute retracts (we refer the reader to~\cite{BaP99} for other characterizations, leading to polynomial-time recognition algorithms):

\begin{theorem}[\cite{BaP99}]\label{thm:absolute-retract}
Let $k \geq 3$. The graph $G=(V,E)$ is an absolute retract of $k$-chromatic graphs if and only if the following conditions hold for any proper $k$-coloring $c : V \to \{1,2,\ldots,k\}$:
\begin{enumerate}
\item For each colour $i$, any family of balls intersects in colour $i$ whenever each pair of them does.
\item Every maximal clique of $G$ has cardinality exactly $k$. 
\item For each colour $i$ and any two non-adjacent vertices $u$ and $v$ such that: $c(v) \neq i$, and either $c(u) \neq i$ or $d_G(u,v) \geq 3$, there is a neighbour $x$ of $v$ on a shortest $uv$-path s.t. $c(x) = i$. 
\end{enumerate}
\end{theorem}

Let $G=(V,E)$ be equipped with a proper $k$-coloring $c$. For every colour $i$, let $V_i := \{ v \in V \mid c(v) = i \}$ be called a colour class. We define, for every $v \in V_i$, $e_i(v) := \max\{ d_G(u,v) \mid u \in V_i \}$. A vertex $v \in V_i$ is $i$-peripheral if it maximizes $e_i(v)$. Finally, let $d_i := \max\{ e_i(v) \mid v \in V_i \}$. We now generalize Lemma~\ref{lem:diam-bipartite}, as follows:

\begin{lemma}\label{lem:diam-kchrom}
Let $G=(V,E)$ be an absolute retract of $k$-chromatic graphs for some $k \geq 3$, and let $c$ be a corresponding proper $k$-coloring. Then, $\max_{1 \leq i \leq k} d_i \leq diam(G) \leq 1 + \max_{1 \leq i \leq k} d_i$. Moreover, if $diam(G) \geq 3$, then we have $diam(G) = 1 + \max_{1 \leq i \leq k} d_i$ if and only if:
\begin{itemize}
\item either $\max_{1 \leq i \leq k} d_i = 2$;
\item or, for some $i \neq j$ s.t. $d_i = d_j$ is maximized, there is some $i$-peripheral vertex whose all neighbours coloured $j$ are $j$-peripheral. 
\end{itemize}
\end{lemma}
\begin{proof}
Clearly, we have $\max_{1 \leq i \leq k} d_i \leq diam(G)$. In order to prove that we also have $diam(G) \leq 1 + \max_{1 \leq i \leq k} d_i$, it suffices to prove that each colour class $V_i$ is a dominating set of $G$. For that, consider any vertex $v \in V \setminus V_i$. Let $K$ be a maximal clique containing $v$. By Theorem~\ref{thm:absolute-retract}, we have $|K| = k$, and therefore, $K \cap V_i \neq \emptyset$. In particular, $N_G(v) \cap V_i \neq \emptyset$, as desired.

In what follows, $diam(G) \geq 3$. Then, if  $\max_{1 \leq i \leq k} d_i = 2$, we must have $diam(G) = 1 + \max_{1 \leq i \leq k} d_i = 3$. From now on, we assume $diam(G) \geq \max_{1 \leq i \leq k} d_i \geq 3$. 

Let us first assume $diam(G) = 1 + \max_{1 \leq i \leq k} d_i$. Consider some peripheral vertex $v$ of $G$, and let $c(v) = i$. Since $V_i$ is a dominating set of $G$, $e_G(v) \leq 1+e_i(v) \leq 1 + d_i$. In particular, we must have $d_i$ is maximized and $v$ is $i$-peripheral (otherwise, $e_G(v) < diam(G)$). We pick some $u \in V$ such that $d_G(u,v) = diam(G)$, and let $c(u) = j$. Note that $j \neq i$ (otherwise, $d_G(u,v) \leq d_i < diam(G)$). Let $x \in N_G(v) \cap V_j$ be arbitrary. Observe that we always have $d_G(u,v) \leq 1 + d_G(u,x) \leq 1 + e_j(x) \leq 1 + d_j$. As a result, $d_j = d_i$ is maximized, and any neighbour $x$ of vertex $v$ with $c(x) = j$ is $j$-peripheral.

Conversely, suppose by contradiction this above necessary condition for having $diam(G) = 1 + \max_{1 \leq i \leq k} d_i$ is not sufficient on its own for the absolute retracts of $k$-chromatic graphs. Without loss of generality, amongst all the absolute retracts of $k$-chromatic graphs, $G$ is a minimum counter-example for which the condition holds and $diam(G) = \max_{1 \leq i \leq k} d_i$. For some $i \neq j$ such that $d_i = d_j$ is maximized, let $v$ be a $i$-peripheral vertex whose all neighbours coloured $j$ are $j$-peripheral, and let $u \in N_G(v)$ be such that $c(u) = j$. Then, $e_G(u) = d_j = d_i = diam(G)$. Observe that $u$ cannot be adjacent to all the vertices coloured $i$ (otherwise, $d_i = 2 < 3$). Therefore, by Theorem~\ref{thm:pesch-characterization}, $u$ is covered and $G' = G \setminus u$ is an absolute retract of $k$-chromatic graphs. Let $V_1',V_2',\ldots,V_k'$ be the colour classes of $G'$ (induced by the coloring $c$ restricted on $V \setminus \{u\}$). Observe that for $p \neq j$ we have $V_p' = V_p$, while $V_j' = V_j \setminus \{u\}$. Furthermore, let $d_1',d_2',\ldots,d_k'$ be the largest distances between vertices in a same colour class of $G'$. Since $G'$ is isometric, $d_p' \leq d_p$ for all $p$. But, since $V_i' = V_i$, we get $d_i' = d_i$, and therefore, $\max_p d_p' = \max_p d_p = d_i = diam(G)$. It implies $diam(G') = diam(G) = \max_p d_p'$. By minimality of $G$, the subgraph $G'$ is not a counter-example. Hence, vertex $v$ must have some neighbour $u' \neq u$ coloured $j$ s.t. $\max\{ d_G(u',x) \mid x \in V_j \setminus \{u\} \} \leq d_i-1$. However, it implies $e_j(u') \leq \max\{d_G(u',u),d_i - 1\} = \max\{2,d_i-1\} = d_i-1$. A contradiction.   
\end{proof}

\subsection{Computation of the $d_i$'s}\label{sec:di}

Our strategy is as follows. First, we prove in Lemma~\ref{lem:reduction-3} that we can reduce our study to the case $k=3$. In Lemmas~\ref{lem:small-diam} and~\ref{lem:large-diam-2}, respectively, we deal with the cases $d_i = {\cal O}(\sqrt{n})$ and $d_i = \Omega(\sqrt{n})$, respectively.

\paragraph{Reduction.}
The next Lemma~\ref{lem:reduction-3} allows us to reduce the general case to $k=3$.

\begin{lemma}\label{lem:reduction-3}
Let $G=(V,E)$ be an absolute retract of $k$-chromatic graphs for some $k \geq 3$, and let $c$ be a corresponding proper $k$-coloring.
For every distinct colours $i_1,i_2,i_3$, the subgraph $H := G[V_{i_1} \cup V_{i_2} \cup V_{i_3}]$ is isometric. Moreover, $H$ is an absolute retract of $3$-chromatic graphs. 
\end{lemma}

\begin{proof}
Suppose by contradiction $H$ is not isometric in $G$.
Let $u,v \in V(H)$ be such that no shortest $uv$-path of $G$ is contained in $H$, and $d_G(u,v)$ is minimum for this property. Since $H$ is induced, $d_G(u,v) > 1$. But then, let $j \in \{i_1,i_2,i_3\} \setminus \{c(u),c(v)\}$. By Theorem~\ref{thm:absolute-retract}, there is a neighbour $x$ of $v$ that is on a shortest $uv$-path of $G$ and such that $c(x) = j$. In particular, $x \in V(H)$. Since $H$ is induced, we have $vx \in E(H)$. However, by minimality of $d_G(u,v)$ we also have $d_H(x,u) = d_G(x,u)$, and therefore $d_H(u,v) \leq 1 + d_H(x,u) = 1 + d_G(x,u) = d_G(v,u)$. A contradiction.

Now, let us prove that $H$ is an absolute retract of $3$-chromatic graphs. For that, it suffices to prove that $H$ satisfies the three properties stated in Theorem~\ref{thm:absolute-retract}. -- Note that, since $c$ is the unique proper $k$-coloring of $G$, the restriction of $c$ to $H$ is the unique proper $3$-coloring of this subgraph. --
\begin{enumerate}
\item Let $j \in \{i_1,i_2,i_3\}$ be fixed. Consider some family of balls in $H$, denoted $N_H^{r_1}[v_1], N_H^{r_2}[v_2],\ldots,\\N_H^{r_q}[v_q]$. Let us assume these balls pairwise intersect in colour $j$, {\it i.e.}, for every $1 \leq a,b \leq q$ we have $N_H^{r_a}[v_a] \cap N_H^{r_b}[v_b] \cap V_j \neq \emptyset$. Since $H$ is a subgraph of $G$, the balls $N_G^{r_1}[v_1], N_G^{r_2}[v_2],\ldots,N_G^{r_q}[v_q]$ also pairwise intersect in colour $j$ (in $G$). By Theorem~\ref{thm:absolute-retract}, there exists a vertex $z_j \in V_j \cap \left( \bigcap \{ N_G^{r_a}[v_a] \mid 1 \leq a \leq q \}\right)$. Since $V_j \subseteq V(H)$, and $H$ is isometric in $G$, we get $z_j \in V_j \cap \left( \bigcap \{ N_H^{r_a}[v_a] \mid 1 \leq a \leq q \}\right)$.  
\item Let $K$ be a maximal clique of $H$. We have $K \subseteq K'$ where $K'$ is a maximal clique of $G$. By Theorem~\ref{thm:absolute-retract}, $|K'| = k$. It implies that every colour class intersects $K'$. Hence, $|K| = |K' \cap (V_{i_1} \cup V_{i_2} \cup V_{i_3})| = 3$. 
\item Finally, let $u,v \in V(H)$ and let $j \in \{i_1,i_2,i_3\}$ be such that $c(v) \neq j$ and either $c(u) \neq j$ or $d_H(u,v) \geq 3$. Since $H$ is isometric, either $c(u) \neq j$ or $d_G(u,v) \geq 3$. By Theorem~\ref{thm:absolute-retract}, there exists a neighbour $x$ of $v$ coloured $j$ that is on a shortest $uv$-path in $G$. Since $x \in V(H)$ and $H$ is isometric, $x$ is also on a shortest $uv$-path in $H$. 
\end{enumerate}
Overall, all three properties of Theorem~\ref{thm:absolute-retract} are satisfied by $H$. It implies that $H$ is an absolute retract of $3$-chromatic graphs. 
\end{proof}

\paragraph{Small-diameter case.} Next, we deal with the case when $d_i$ is sufficiently small.

\begin{lemma}\label{lem:small-diam}
Let $G=(V,E)$ be an absolute retract of $3$-chromatic graphs, and let $c$ be a corresponding proper $3$-coloring. For each colour $i$ and $D \geq 2$, we can compute in ${\cal O}(Dm)$ time the set $U_i := \{ v \in V_i \mid e_i(u) \leq D \}$.   
\end{lemma}
\begin{proof}
It suffices to prove the result for $i=1$.
For that, we generalize the techniques used in Lemma~\ref{lem:cst-ecc}.
Specifically, for every $1 \leq t \leq D$, the objective is to compute a partition of $V_1$, denoted ${\cal P}_t$, subject to the following property: for each colour $i$, the $|{\cal P}_t|$ subsets $B_t[A,i] := V_i \cap \left( \bigcap \{ N_G^t[a] \mid a \in A \} \right)$, for all $A \in {\cal P}_t$, must be pairwise disjoint and, if $i \neq 1$ or $t \geq 2$, these sets must also be nonempty. Being given any such partition ${\cal P}_D$, we observe that $U_1 \neq \emptyset$ if and only if ${\cal P}_D = (V_1)$ and, if it is the case, then $U_1 = B_D[V_1,1]$.  

\smallskip
{\bf The algorithm.} For the base case $t=1$, let $Y := V_1$. While $Y \neq \emptyset$, we select some vertex $x \in V_2$ such that $|N_G(x) \cap Y|$ is maximized, we add $N_G(x) \cap Y$ as a new group in ${\cal P}_1$, then we set $Y := Y \setminus N_G(x)$. Finally, we enumerate all groups $A \in {\cal P}_1$, and we scan the neighbourhoods of all vertices $a \in A$ in order to compute the sets $B_1[A,i]$ for each colour $i$ (note that $B_1[A,1] = A$ if $|A|=1$ and $B_1[A,1] = \emptyset$ otherwise). 

At step $t+1$, for each group $A \in {\cal P}_t$, let us define:
$$W(A,1) := N_G(B_t[A,2]) \cap V_1,$$
$$W(A,2) := N_G(B_t[A,3]) \cap V_2,$$
$$W(A,3) := N_G(B_t[A,2]) \cap V_3.$$ 
Let ${\cal F} := {\cal P}_t$. While ${\cal F} \neq \emptyset$, we proceed as follows. We pick some $v \in V_1$ such that $\#\{ A \in {\cal F} \mid v \in W(A,1) \}$ is maximized. Let $A' := \bigcup \{ A \in {\cal F} \mid v \in W(A,1) \}$. We add $A'$ to ${\cal P}_{t+1}$, for each colour $i$ we set $B_{t+1}[A',i] := \bigcap \{ W(A,i) \mid A \in {\cal F}, \ v \in W(A,1) \}$, and we set ${\cal F} := {\cal F} \setminus \{ A \in {\cal F} \mid v \in W(A,1) \}$. 

\smallskip
{\bf Correctness.} We first consider the base case $t=1$. Clearly, the sets $B_1[A,1]$ (equal to either $A$, if it is a singleton, and to the empty set otherwise) are pairwise disjoint. Since at every step, in order to create a new group, we pick a vertex of $V_2$ with a maximum number of unselected vertices in $V_1$, it follows by construction that the sets $B_1[A,2]$ are nonempty. By maximality of the vertex $V_2$ selected at each step, it also follows that all these sets $B_1[A,2]$ are pairwise disjoint (this is exactly the same argument as the one used for Claim~\ref{claim:k=2}, in Lemma~\ref{lem:cst-ecc}). The next claim shows that both properties ({\it i.e.}, being nonempty and pairwise disjoint) also hold for the sets $B_1[A,3]$.  
\begin{myclaim}\label{claim-1}
For $A \subseteq V_1$, the neighbour-sets $N_G(a), \ a \in A$ intersect in colour $2$ if and only if they intersect in colour $3$.
\end{myclaim}
\begin{proofclaim}
Let us assume the balls $N_G[a], \ a \in A$ intersect in colour $2$ (resp., in colour $3$). In particular, the vertices of $A$ are pairwise at distance two. By the third property of Theorem~\ref{thm:absolute-retract}, the balls $N_G[a], \ a \in A$ pairwise intersect in colour $3$ (resp., in colour $2$), and therefore the claim now follows from the first (Helly-type) property of this Theorem~\ref{thm:absolute-retract}. 
\end{proofclaim}
From now on, we consider the inductive step ({\it i.e.}, from $t$ to $t+1$). 
We need the following claim for our analysis:
\begin{myclaim}\label{claim-3}
For each colour $i$, $t \geq 1$ and $A \in {\cal P}_t$, we have $W(A,i) = V_i \cap \left( \bigcap \{ N_G^{t+1}[a] \mid a \in A \mid \} \right)$.
\end{myclaim}
\begin{proofclaim}
By construction, $W(A,i)\subseteq V_i \cap \left( \bigcap \{ N_G^{t+1}[a] \mid a \in A \} \right)$. Conversely, let $u \in \bigcap \{ N_G^{t+1}[a] \mid a \in A \}$ be such that $c(u) = i$. Let $j$ be the least colour available amongst $\{2,3\} \setminus \{i\}$ ({\it i.e.}, $j = 2$ if $i \neq 2$ and $j = 3$ otherwise). We prove as a subclaim that, for every $a \in A$, the balls $N_G^t[a]$ and $N_G[u]$ intersect in colour $j$. For that, we need to consider three cases:
\begin{itemize}
\item\underline{Case $u = a$.} This can only happen if $i = 1$. Then, $N_G[u] \subseteq N_G^t[u] = N_G^t[a]$. Since $u$ has at least one neighbour coloured $j$, the balls $N_G^t[a]$ and $N_G[u]$ trivially intersect in colour $j$. 
\item\underline{Case $ua \in E$.} By Theorem~\ref{thm:absolute-retract}, every maximal clique of $G$ has cardinality $3$. It implies that $u,a$ have a common neighbour coloured $j$. 
\item\underline{Case $d_G(u,a) \geq 2$.} By Theorem~\ref{thm:absolute-retract}, there exists a neighbour $x$ of $u$ on a shortest $au$-path such that $c(x) = j$. 
\end{itemize} 
Furthermore, since $B_t[A,j] \neq \emptyset$, the balls $N_G^t[a], \ a \in A$, also pairwise intersect in colour $j$. By Theorem~\ref{thm:absolute-retract}, there exists a $x \in N_G(u) \cap B_t[A,j]$. In order to conclude the proof, we just need to observe that $u \in N_G(B_t[A,j]) \cap V_i = W(A,i)$.
\end{proofclaim}
When a new set $A'$ is created, as the union of sets in some nonempty family ${\cal Q} \subseteq {\cal P}_t$, we define for each colour $i$, $B_{t+1}[A',i] = \bigcap \{ W(A,i) \mid A \in {\cal Q} \}$. By the above Claim~\ref{claim-3}, $B_{t+1}[A',i] = V_i \cap \left( \bigcap \{ N_G^{t+1}[a'] \mid a' \in A' \mid \} \right)$, as desired. It now remains to prove that all sets $B_{t+1}[A',i], \ A' \in {\cal P}_{t+1}$ are nonempty and pairwise disjoint. Observe that in order to create a new group, we repeatedly pick a vertex $v \in V_1$ with a maximum number of unselected groups $A \in {\cal P}_t$ such that $v \in W(A,1)$. To ensure that all groups $A \in {\cal P}_t$ are eventually selected, we prove that:
\begin{myclaim}\label{claim-2}
For each $t \geq 1$ and $A \in {\cal P}_t$, we have $W(A,1) \neq \emptyset$.
\end{myclaim}
\begin{proofclaim}
Let $x \in B_t[A,2]$. Since $G$ is an absolute retract, $N_G(x) \cap V_1 \neq \emptyset$. We are now done as we have by construction $N_G(x) \cap V_1 \subseteq W(A,1)$. 
\end{proofclaim}
By maximality of the vertex of $V_1$ selected at each step, the sets $B_{t+1}[A',1], \ A' \in {\cal P}_{t+1}$ are nonempty and pairwise disjoint. Finally, the next claim shows that both properties also hold for the sets $B_{t+1}[A',i]$, for any colour $i$:
\begin{myclaim}\label{claim-4}
For each colour $i \neq 1$, $A \subseteq V_1$ and $t \geq 2$, the balls $N_G^t[a], \ a \in A$ intersect in colour $1$ if and only if they also intersect in colour $i$.
\end{myclaim}
\begin{proofclaim}
First assume that the balls $N_G^t[a], \ a \in A$ intersect in colour $1$. Let $a,a' \in A$ and let $v \in N_G^t[a] \cap N_G^t[a']$ such that $c(v) = 1$. W.l.o.g., $a \neq v$. For $j \in \{2,3\} \setminus \{i\}$, by Theorem~\ref{thm:absolute-retract} there exists a neighbour $x$ of $v$ on a shortest $av$-path such that $c(x) = j$. We divide our analysis into three cases:
\begin{itemize}
\item\underline{Case $xa' \notin E$.} By Theorem~\ref{thm:absolute-retract}, there exists a neighbour $y$ of $x$ on a shortest $xa'$-path such that $c(y) = i$. Note that $y \in N_G^t[a] \cap N_G^t[a']$.
\item\underline{Case $xa,xa' \in E$.} By Claim~\ref{claim-1}, the neighbour-sets $N_G(a), \ N_G(a')$ intersect in colour $i$. Since $t \geq 2$, so do the balls $N_G^t[a],\ N_G^t[a']$.
\item\underline{Case $xa \notin E, \ xa' \in E$.} By Theorem~\ref{thm:absolute-retract}, there exists a neighbour $y$ of $x$ on a shortest $ax$-path such that $c(y) = i$. Note that $y \in N_G^{t-2}[a] \cap N_G^2[a'] \subseteq N_G^t[a] \cap N_G^t[a']$.
\end{itemize}
Overall, the balls $N_G^t[a], \ a \in A$ pairwise intersect in colour $i$, and therefore by Theorem~\ref{thm:absolute-retract} there is a vertex coloured $i$ in their common intersection.

Conversely, assume that the balls $N_G^t[a], \ a \in A$ intersect in colour $i$. Let $a,a' \in A$ and let $u \in N_G^t[a] \cap N_G^t[a']$ such that $c(u) = i$. Here also, we divide our analysis into several cases:
\begin{itemize}
\item\underline{Case $ua,ua' \notin E$.} Let $j \in \{2,3\} \setminus \{i\}$. By Theorem~\ref{thm:absolute-retract} (applied twice) there exist $x,y \in N_G(u)$ that are on a shortest $au$-path and a shortest $a'u$-path respectively, such that $c(x) = c(y) = j$. If $x = y$ then, any neighbour $z \in N_G(x)$ coloured $1$ is in $N_G^t[a] \cap N_G^t[a']$. Otherwise, $d_G(x,y) = 2$. By Theorem~\ref{thm:absolute-retract}, there is a shortest $xy$-path whose internal node $z$ has colour $1$. Here also, $z \in N_G^t[a] \cap N_G^t[a']$.
\item\underline{Case $ua,ua' \in E$.} In this situation, $d_G(a,a') = 2$, and the balls $N_G^2[a] \subseteq N_G^t[a]$ and $N_G^2[a'] \subseteq N_G^t[a']$ intersect in colour $1$.
\item\underline{Case $ua \in E, \ ua' \notin E$} (The case $ua \notin E, \ ua' \in E$ is symmetrical to this one). If $d_G(a,a') \leq t$ then, the balls $N_G^t[a]$ and $N_G^t[a']$ intersect in colour $1$. Otherwise, $d_G(a',u) = t$ and $d_G(a,a') = t+1$. If furthermore $d_G(u,a') = t \geq 3$ then, by Theorem~\ref{thm:absolute-retract}, there exists a neighbour $x$ of $u$ on a shortest $a'u$-path such that $c(x) = 1$. In this situation, $x \in N_G^2[a] \cap N_G^{t-1}[a'] \subseteq N_G^t[a] \cap N_G^t[a']$. Thus, we are left considering the special subcase $t = 2$. Let $j \in \{2,3\} \setminus \{i\}$, and let $v \in N_G(a') \cap N_G(u)$. Observe that we have $c(v) = j$. Furthermore, by Theorem~\ref{thm:absolute-retract}, vertices $a$ and $u$ must have a common neighbour $y$ coloured $j$ ({\it i.e.}, the third vertex in a maximal clique containing $a$ and $u$). By construction, $d_G(y,v) = 2$. Then, again by Theorem~\ref{thm:absolute-retract}, there exists a common neighbour $z \in N_G(y) \cap N_G(v)$ such that $c(y) = 1$. In this situation, $y \in N_G^2[a] \cap N_G^2[a'] \subseteq N_G^t[a] \cap N_G^t[a']$.  
%Let $j \in \{2,3\} \setminus \{i\}$. By Theorem~\ref{thm:absolute-retract}, there exists a neighbour $x$ of $a$ on a shortest $aa'$-path such that $c(x) = j$. Then, by a second application of Theorem~\ref{thm:absolute-retract}, there exists a neighbour $y$ of $x$ on a shortest $ax$-path such that $c(y) = i$. If $y \in N_G(a')$ then, we get $y \in N_G^2[a] \cap N_G[a'] \subseteq N_G^t[a] \cap N_G^t[a']$. Otherwise, by replacing $u$ with $y$, we are back to the first Case above. 
\end{itemize} 
Overall, the balls $N_G^t[a], \ a \in A$ pairwise intersect in colour $1$, and therefore by Theorem~\ref{thm:absolute-retract} there is a vertex coloured $1$ in their common intersection.
\end{proofclaim}
\smallskip
{\bf Complexity.} We revisit the approach taken in Lemma~\ref{lem:cst-ecc}. For the base case $t = 1$ we compute in ${\cal O}(m)$ time, for each vertex of $V_2$, its number of neighbours in $V_1$. We create an array of $|V_1|$ lists, numbered from $1$ to $|V_1|$. Each vertex $u \in V_2$ is put in the list numbered $|N_G(u) \cap V_1|$. It takes ${\cal O}(n)$ time. Then, we scan the lists in decreasing order ({\it i.e.}, from $|V_1|$ downto $1$), going to the next list each time the current one is empty. When the current list is nonempty, we pick any vertex $u$ of this list in order to create the next subset $A \in {\cal P}_1$. Note that $A$ can be constructed in ${\cal O}(|N_G(u)|)$ time. Since each vertex $u$ gets used for the creation of at most one group (otherwise, there would exist $A,A' \in {\cal P}_1$ such that $B_1[A,2] \cap B_1[A',2] \neq \emptyset $), the total running time in order to create all the groups of ${\cal P}_1$ is in ${\cal O}(n+m)$. Furthermore, after a group $A$ is created, we need to actualize the number of unselected vertices in $V_1$ for each $x \in V_2$ (discarding all such vertices whose all neighbours in $V_1$ are already selected). For that, we first scan $N_G(a), a \in A$ in order to update, for each $x \in V_2$, its number of unselected vertices in $V_1$. Then, if each vertex of $V_2$ keeps a pointer to its position in the unique list in which it is contained, we can update the list contents in ${\cal O}(\sum_{a \in A}|N_G(a)|)$ time (maximum number of vertices of $V_2$ that need to be moved). Overall, this phase also takes total ${\cal O}(n+m)$ time. Finally, since ${\cal P}_1$ is a partition of $V_1$, for each colour $i$ we can create the sets $B_1[A,i], i \in A$ in ${\cal O}(m)$ time, simply by scanning the neighbour-sets of each vertex of $V_1$.  
In order to complete our complexity analysis, it remains to prove that each inductive step (from $t$ to $t+1$) also requires ${\cal O}(m)$ time. 
For that, since by the hypothesis the sets $B_t[A,i], \ A \in {\cal P}_t$ are pairwise disjoint, we can compute the sets $W(A,i)$ in total linear time. Then, we proceed as for the first case (see also Lemma~\ref{lem:cst-ecc}), but we now consider the vertices of $V_1$ rather than $V_2$, and for these vertices, we keep track of the number of sets $W(A,i)$ they belong to rather than keeping track of their degrees. We get a running time proportional to $\sum_{A \in {\cal P}_t} W(A,i) = {\cal O}(n+m)$. 
\end{proof}

\paragraph{Large-diameter case.} Finally, we address the case when $d_i$ is large. For that, we start with a simple intermediate lemma.
Recall that, for every two vertices $u$ and $v$ and any $\ell \leq d(u,v)$, we can define the slice $L(u,\ell,v) := \{ x \in I(u,v) \mid d(u,x) = \ell \}$.

\begin{lemma}\label{lem:coloured-neighbour}
Let $G=(V,E)$ be an absolute retract of $k$-chromatic graphs for some $k \geq 3$, and let $c$ be a corresponding proper $k$-coloring.
For each colour $i$ and $u,v \in V_i$ such that $u \neq v$ and $d_G(u,v) \neq 3$, there exists a $x \in L(u,2,v)$ coloured $i$.
\end{lemma}
\begin{proof}
If $d_G(u,v) = 2$ then, $x = v$. Otherwise, $d_G(u,v) \geq 4$. By Theorem~\ref{thm:absolute-retract}, for any $j \neq i$ there exists a neighbour $y$ of $u$ coloured $j$ on a shortest $uv$-path. Since $d_G(y,v) \geq 3$, again by Theorem~\ref{thm:absolute-retract}, there exists a neighbour $x$ of $y$ coloured $i$ on a shortest $yv$-path, and so, on a shortest $uv$-path.
\end{proof}

A function is called unimodal if every local minimum is also a global minimum. It is known that the eccentricity function of a Helly graph is unimodal~\cite{Dra89}, and this property got used in~\cite{DDG19+} in order to compute all the eccentricities in this graph class in subquadratic time. Next, we prove that a similar, but weaker property holds for each colour class of absolute retracts, namely:

\begin{lemma}\label{lem:almost-unimodal}
Let $G=(V,E)$ be an absolute retract of $k$-chromatic graphs for some $k \geq 3$, and let $c$ be a corresponding proper $k$-coloring.
For each colour $i$ and any $u \in V_i$ s.t. $e_i(u) \geq (d_i+5)/2 \geq 7$, there exists a $u' \in V_i$ s.t. $d_G(u,u') = 2$ and $e_i(u') = e_i(u)-2$.
\end{lemma}

\begin{proof}
Let $X_i = \{ x \in V_i \mid d_G(u,x) \geq 4 \}$. By Lemma~\ref{lem:coloured-neighbour}, for every $x \in X_i$, the balls $N_G^2[u]$ and $N_G^{e_i(u)-2}[x]$ intersect in colour $i$. Furthermore, let $x,x' \in X_i$ be arbitrary. If $d_G(x,x') \leq e_i(u) - 2$ then the balls $N_G^{e_i(u)-2}[x]$ and $N_G^{e_i(u)-2}[x']$ intersect in colour $i$. Otherwise, let $\ell := d_G(x,x') - (e_i(u)-2)$, and let $2t \in \{\ell,\ell+1\}$ be even. We prove by finite induction the existence of vertices $x_0,x_1,\ldots,x_t \in I_G(x,x')$ s.t., for every $0 \leq p \leq t$, we have $d_G(x,x_p) = 2p$ and $c(x_p) = i$. For the base case $p=0$, we have $x_0 = x$. For $0 < p \leq t$, we observe that: 
\begin{align*}
d_G(x_{p-1},x') &= d_G(x,x') - 2(p-1) = d_G(x,x') - 2p + 2 \\
&\geq d_G(x,x') - 2t + 2  \geq d_G(x,x') - (\ell+1) + 2 \\
&= d_G(x,x') - \ell + 1 = e_i(u)-1 > 3.
\end{align*}
Then, the existence of $x_p$ follows from Lemma~\ref{lem:coloured-neighbour} (applied to the pair $x_{p-1},x'$).
Furthermore, 
\begin{align*}
2t &\leq \ell + 1 = d_G(x,x') - e_i(u) + 3 \leq d_i - e_i(u) + 3 \\
&\leq d_i - (d_i+5)/2 + 3 = (d_i+1)/2 = (d_i+5)/2 -2 \\
&\leq e_i(u) -2.
\end{align*}
%\geq 6 > 3$. By Lemma~\ref{lem:coloured-neighbour}, there exists a $y \in L(x,2,x')$ s.t. $c(y) = i$. Note that in this case, $y \in N_G^{e_i(u)-2}[x] \cap N_G^{e_i(u)-2}[x']$.
Thus, in this situation, $x_{2t} \in N_G^{e_i(u)-2}[x] \cap N_G^{e_i(u)-2}[x']$.  
Overall, we obtain that the balls $N_G^2[u]$ and $N_G^{e_i(u)-2}[x], \ x \in X_i$ pairwise intersect in colour $i$. By Theorem~\ref{thm:absolute-retract}, there exists a $u' \in N_G^2[u] \cap \left( \bigcap \{ N_G^{e_i(u)-2}[x] \mid x \in X_i \} \right)$ such that $c(u') = i$. Note that $u' \neq u$ because there must be a $x \in X_i$ s.t. $d_G(u,x) = e_i(u) \geq 4$. We obtain that $d_G(u,u') = 2$ and that, for every $x \in V_i \setminus X_i$ we have $d_G(u',x) \leq 2 + d_G(u,x) \leq 5$. As a result, $e_i(u') \leq \max\{5,e_i(u)-2\} = e_i(u) -2$. Since $d_G(u,u') = 2$, this must be an equality.  
\end{proof}

We end up applying this almost-unimodality property to the computation of the $d_i$'s (assuming these values to be at least in $\Omega(\sqrt{n})$):

\begin{lemma}\label{lem:large-diam-2}
Let $G=(V,E)$ be an absolute retract of $k$-chromatic graphs for some $k \geq 3$, let $c$ be a corresponding proper $k$-coloring, and let $i$ be such that $d_i \geq 8D+5 = \omega(\log{n})$. Then, with high probability, we can compute in total $\tilde{\cal O}(mn/D)$ time the value $d_i$ and the $i$-peripheral vertices.
\end{lemma}
\begin{proof}
We modify Algorithm~\ref{alg:large-diam} as follows. 

\smallskip
\noindent
{\bf The algorithm.} We set $p = \alpha \frac{\log{n}}{D}$, for some sufficiently large constant $\alpha$. Then, let $U_i(p)$ contain every $v \in V_i$ independently with probability $p$. For all $v \in V_i$, let $\bar{e}_i(v) := \min \{ d_G(u,v) + e_i(u) \mid u \in U_i(p), d_G(u,v) \leq D \}$ (with the understanding that, if no vertex of $U_i(p)$ is at a distance $\leq D$ from $v$, then $\bar{e}_i(v) = 0$). We output $d_i = \max\{ \bar{e}_i(v) \mid v \in V_i \}$, and we identify as $i$-peripheral all vertices $v \in V_i$ such that $\bar{e}_i(v)$ is maximized. 

\smallskip
\noindent
{\bf Complexity.} Since we have to perform a breadth-first search from every vertex of $U_i(p)$, the running time of the algorithm is in ${\cal O}(m|U_i(p)|)$. Let us prove that with high probability, $|U_i(p)| = \tilde{\cal O}(n/D)$. First note that we have $\mathbb{E}\left[ \ |U_i(p)| \ \right] =  \tilde{\Theta}(|V_i|/D)$. We claim to have $|V_i| = \Omega(D)$. In order to see that, consider $x,y \in V_i$ s.t. $d_G(x,y) = d_i$. Recall that $d_i = \Omega(D)$. Thus, by applying Lemma~\ref{lem:coloured-neighbour} $\Omega(D)$ times, we get the existence of a shortest $xy$-path with $\Omega(D)$ vertices coloured $i$. The latter proves, as claimed, $|V_i| = \Omega(D)$. Then, $\mathbb{E}\left[ \ |U_i(p)| \ \right] = \Omega(\log{n})$. By Chernoff bounds, we have $|U_i(p)| = \tilde{\cal O}(|V_i|/D) = \tilde{\cal O}(n/D)$ with high probability. 

\smallskip
\noindent
{\bf Correctness.} Let $v \in V_i$ be arbitrary. We divide our analysis into two cases.
\begin{itemize}
\item\underline{Case $e_i(v) < (d_i+5)/2 + 2D$.} For any $u \in U_i(p)$ s.t. $d_G(u,v) \leq D$, $e_i(u) < (d_i+5)/2 + 3D$. Therefore, $\bar{e}_i(v) < (d_i+5)/2 + 4D = d_i/2 + (8D+5)/2 \leq d_i$. 
\item\underline{Case $e_i(v) \geq (d_i+5)/2 + 2D$.} In particular, we always fall in this case if $v$ is $i$-peripheral. Then, by applying Lemma~\ref{lem:almost-unimodal} $D$ times, we get the existence of vertices $x_0 = v, x_1,x_2,\ldots,x_D$ such that, for every $j > 0$, $d_G(x_{j-1},x_j) = 2$ and $e_i(x_j) = e_i(x_{j-1}) - 2$. Observe that we have $d_G(v,x_j) \leq 2j$ and $e_i(x_j) = e_i(v) - 2j$, and therefore $d_G(v,x_j) = 2j$. It implies that, if $x_j \in U_i(p)$ for some $j$, $\bar{e}_i(v) = e_i(v)$. Let us prove this happens with high probability. Indeed:
\begin{align*}
\mathbb{P}r[ U_i(p) \cap \{x_0,x_1,\ldots,x_D \} = \emptyset] < (1-p)^{D} = (1-p)^{\frac{\alpha\log{n}}p} \leq n^{-\alpha}.
\end{align*} 
\end{itemize}
Summarizing both cases above, for all $i$-peripheral vertices $v$ we have $\bar{e}_i(v) = e_i(v) = d_i$, whereas for every $v \in V_i$ that is not $i$-peripheral, $\bar{e}_i(v) < d_i$. 
\end{proof}

%\begin{lemma}\label{lem:cst-ecc-kchrom}
%Let $G=(V,E)$ be an absolute retract of $k$-chromatic graphs for some $k \geq 3$, and let $c$ be a corresponding proper $k$-coloring.
%For any integer $\ell \geq 2$, we can compute in total ${\cal O}(\ell m)$ time the subsets $U_i := \{ v \in V_i \mid e_i(v) \leq \ell\}$, for each colour $i$.
%\end{lemma}
%\begin{proof}
%Fix two colour $i \neq j$. We now explain how to compute $U_i$. 
%\begin{enumerate}
%\item Let $G_{i,j} := G[V_i \cup V_j]$ be the subgraph induced by the corresponding two colour classes. Let $2t$ be the largest even integer smaller than $\ell$ (note that $2t \in \{\ell-1,\ell\}$). 
%\item We apply the procedure of Lemma~\ref{lem:cst-ecc}. Specifically, we repeatedly compute partitions ${\cal P}_q, \ 0 \leq q \leq 2t$ of $V_i$, along with for every group $A \in {\cal P}_q$ a corresponding set $B^A$ (to be specified below). The computation goes as follows. Initially, we set ${\cal P}_0 := V_i$ (each group is a singleton) and, for every group $\{v\} \in {\cal P}_0$ ({\it i.e.}, for every $v \in V_i$) we set $B^{\{v\}} = \{v\}$ as the corresponding set. Then, at step $q > 0$, we construct ${\cal P}_q$ and its corresponding sets from the outcome of the previous step.
%\end{enumerate}
%
%For that,  We apply the 
%\end{proof}

%\subsection{Case $k \geq 4$}\label{sec:k-gq-4}
%
%\subsection{Case $k=3$}\label{sec:k=3}

\subsection{Main result}\label{sec:main-k}

We end up gathering all previous results in this section, in order to prove the following:

\begin{theorem}\label{thm:main-k}
If $G=(V,E)$ is an absolute retract of $k$-chromatic graphs, for some $k \geq 3$, then we can compute its diameter with high probability in $\tilde{\cal O}(m\sqrt{n})$ time.
\end{theorem}

\begin{proof}
We may assume $diam(G) > 1$ (trivial case).
By Corollary~\ref{cor:diam-2}, we can decide in linear time whether $diam(G) \leq 2$.
Thus, from now on, let us assume $diam(G) \geq 3$. 
We compute a proper $k$-coloring of $G$, that takes linear-time according to Proposition~\ref{prop:col}.
Then, by Lemma~\ref{lem:diam-kchrom}, we can compute $diam(G)$ in linear time if, for each colour $i$, we are given $d_i$ and the corresponding $i$-peripheral vertices. 
In order to compute this information, let $V_1,V_2,\ldots,V_k$ be the colour classes.
Up to repeating $V_1$ ($=V_{k+1}$) and $V_2$ ($=V_{k+2}$) at most once we may assume the number of colour to be a multiple of three, and then we partition the colour classes in disjoint triples $V_i \cup V_{i+1} \cup V_{i+2}$.
By Lemma~\ref{lem:reduction-3}, each subgraph $H_i := G[V_i \cup V_{i+1} \cup V_{i+2}]$ is isometric in $G$ and is an absolute retract of $3$-chromatic graphs.
Therefore, we may restrict ourselves to $H_i$ in order to compute the values $d_i,d_{i+1},d_{i+2}$. 
Let $n_i := |V(H_i)|$ and $m_i := |E(H_i)|$.
By construction, $\sum_i n_i = \Theta(n)$ and $\sum_i m_i = \Theta(m)$. 
Hence, if for each $i$ our computations require $\tilde{\cal O}(m_i\sqrt{n_i})$ time, the total running time is in $\tilde{\cal O}(m\sqrt{n})$. Let us prove it is the case for $H_1$ (and so, by symmetry, for every $i$). Since there are only three colour classes in $H_1$, it is sufficient to prove that we can compute $d_1$ and all the $1$-peripheral vertices in $\tilde{\cal O}(m_1\sqrt{n_1})$ time. For that, we first compute $D := e_1(v)$ for some arbitrary $v \in V_1$. It takes ${\cal O}(m_1)$ time and, by the triangular inequality, it is a $2$-approximation of $d_1$. There are now two cases. If $D \leq 16\sqrt{n_1} + 10$ then, we compute by one-sided binary search the smallest $t$ such that $\forall v \in V_1, \ e_1(v) \leq t$, that is exactly $d_1$. Note that at each step of the binary search we apply Lemma~\ref{lem:small-diam}, that results in a running time in $\tilde{\cal O}(m_1d_1) = \tilde{\cal O}(m_1D) = \tilde{\cal O}(m_1\sqrt{n_1})$. Furthermore, we can compute all the $1$-peripheral vertices within the same amount of time, simply with one more call to Lemma~\ref{lem:small-diam} (for $d_1 -1$). Otherwise, $d_1 \geq D/2 \geq 8\sqrt{n_1} + 5$, and we apply Lemma~\ref{lem:large-ecc}. With high probability, the running time is in $\tilde{\cal O}(m_1\sqrt{n_1})$.
\end{proof}

\section{Split graphs}\label{sec:split}

The graphs studied in the two previous Sec.~\ref{sec:k=2} and~\ref{sec:k-chromatic} are exactly the absolute retracts of the class of all (irreflexive) graphs. In this section, we show that by considering the absolute retracts of more restricted graph classes, we may derive negative results about faster diameter computation. Specifically, $G=(V,E)$ is a split graph if its vertex-set $V$ can be partitioned into a clique $K$ and a stable set $S$. In what follows, we always denote such a bipartition by $K + S$. We also use the standard notations $\omega(G)$ and $\alpha(G)$ for, respectively, the clique number (maximum cardinality of a clique) and the independence number (maximum cardinality of a stable set) of $G$. 

Let us recall the following hardness result about the diameter problem on split graphs:

\begin{lemma}[\cite{BCH16}]\label{lem:ov}
For any $\epsilon > 0$, there exists a $c(\epsilon)$ s.t., under SETH, we cannot compute the diameter in ${\cal O}(n^{2-\epsilon})$ time on the split graphs of order $n$ and clique-number at most $c(\epsilon) \log{n}$.
\end{lemma}

Our main result in this section is that this above hardness result also holds for the absolute retracts of split graphs (Theorem~\ref{thm:hardness-absolute-split-retract}). For that, we need a few preparatory lemmas, namely: 

\begin{lemma}[\cite{Gol04}]\label{lem:split-dec}
Let $G=(K + S,E)$ be a split graph. Exactly one of the following conditions holds:
\begin{enumerate}
\item $|K| = \omega(G)$ and $|S| = \alpha(G)$ (in this case the partition $K+S$ is unique);
\item $|K| = \omega(G) - 1$ and $|S| = \alpha(G)$ (in this case there exists an $y \in S$ s.t. $K+\{y\}$ is complete);
\item $|K| = \omega(G)$ and $|S| = \alpha(G)-1$ (in this case there exists an $x \in K$ s.t. $S+\{x\}$ is stable).
\end{enumerate}
\end{lemma}

Recall that $G=(V,E)$ is a complete split graph if there exists a partition of its vertex-set into a clique $K$ and a stable set $S$ s.t. every vertex of $S$ is adjacent to every vertex of $K$. In particular, a complete split graph has diameter at most two. 

\begin{theorem}[\cite{Kla94}]\label{thm:absolute-split-retract}
A split graph is an absolute retract of split graphs if and only if it is a complete split graph or it has a unique partition of its vertex-set into a clique and a stable set.
\end{theorem}

We are now ready to prove the main result of this section, namely:

\begin{theorem}\label{thm:hardness-absolute-split-retract}
There is a linear-time reduction from the diameter problem on split graphs to the same problem on the absolute retracts of split graphs. 

In particular, for any $\epsilon > 0$, there exists a $c(\epsilon)$ s.t., under SETH, we cannot compute the diameter in ${\cal O}(n^{2-\epsilon})$ time on the absolute retracts of split graphs of order $n$ and clique-number at most $c(\epsilon) \log{n}$.
\end{theorem}

\begin{proof}
Let $G=(V,E)$ be a split graph. First, we check in linear time whether $diam(G) = 0$ ({\it i.e.}, $G$ is a singleton) or $diam(G) = 1$ ({\it i.e.}, $G$ is a complete graph). From now on, let us assume $diam(G) \geq 2$. Since $G$ is a split graph, computing the diameter is equivalent to deciding whether $diam(G) = 3$. For that, we compute a partition of $V$ into a clique $K$ and a stable set $S$. It can be done in linear time~\cite{Gol04}. For every vertex $v$, let $deg(v) := |N(v)|$. We apply the following pruning rules until no more vertex can be removed:
\begin{itemize}
\item If there exists a $x \in K$ s.t. $deg(x) = |K|-1$, then we remove $x$ from $G$.  
\item If there exists a $y \in S$ s.t. $deg(y) = |K|$, then we remove $y$ from $G$.
\end{itemize}
{\bf Complexity.} Let us explain how  both rules above can be applied exhaustively in total linear time.
For that, we maintain an array of $n-1$ lists, numbered from $1$ to $n-1$, so that each vertex of $S$ of degree $i$ must be contained into the $i^{th}$ such list. We proceed similarly for the vertices of $K$ (but in a separate array of lists). Note that such a data structure can be initialized in linear time, simply by computing the degree sequence of the graph. If we further store, for each vertex, a pointer to its position in the unique list in which it is contained, then every time we remove a vertex $v$, we can update the structure in a time proportional to $deg(v)$. The latter results in a total update time in ${\cal O}(m)$. Finally, if at each step we maintain the cardinality $|K|$ of the clique $K$, then in order to check whether one of the two pruning rules applies, we are left testing whether at most two lists of our data structure are nonempty.   

\noindent
{\bf Correctness.} In order to prove correctness of these above pruning rules, let us first define a super-simplicial vertex as any vertex $v$ s.t. $N(v) = K \setminus \{v\}$. Note that if $x \in K$ is s.t. $deg(x) = |K|-1$ then, $N(x) = K \setminus \{x\}$, and so, $x$ is super-simplicial. In the same way if $y \in S$ is s.t. $deg(y) = |K|$ then, $N(y) = K$, and so, $y$ is super-simplicial. Summarizing the above, we can only prune super-simplicial vertices. Now, we claim that, for a split graph $G$ and a super-simplicial vertex $v$ of $G$, we have $diam(G) = 3$ if and only if $diam(G \setminus \{v\}) = 3$. Indeed, since $G \setminus \{v\}$ is an isometric subgraph of $G$, $diam(G\setminus\{v\}) = 3$ implies $diam(G) \geq 3$, hence (since $G$ is a split graph) $diam(G) = 3$. Conversely, if $diam(G) = 3$ then, there exist $y,y' \in S$ s.t. $d(y,y') = 3$. In order to prove that $diam(G \setminus \{v\}) = 3$, it suffices to prove that none of $y$ and $y'$ can be a super-simplicial vertex. That is indeed the case because, since $K$ is a dominating set of $G$, any super-simplicial vertex has eccentricity at most two. 

\smallskip
Finally, let $G'$ be the resulting subgraph after no more vertex can be removed. Observe that we have reduced the computation of $diam(G)$ to the computation of $diam(G')$. We shall prove that there exists a unique partition of $V(G')$ into a clique $K'$ and a stable set $S'$. By Theorem~\ref{thm:absolute-split-retract}, this will imply that $G'$ is an absolute retract of split graphs. Let $K'$ and $S'$ be a partition of $V(G')$ into a clique and a stable set, and suppose for the sake of contradiction this partition is not unique. By Lemma~\ref{lem:split-dec}, either there exists $y \in S'$ s.t. $K'+\{y\}$ is complete, or there exists $x \in K'$ s.t. $S'+\{x\}$ is a stable set. In the former case, $deg(y) = |K'|$, while in the latter case, $deg(x) = |K'|-1$. But then, we could still have applied one of our two pruning rules above, a contradiction.
\end{proof}

\section{Planar graphs}\label{sec:planar}

Our last (non-algorithmic) section is about the absolute retracts of planar graphs
The latter have been characterized in~\cite{Hel74a}, assuming the Four-color conjecture. Since this is now a theorem~\cite{ApH76}, the characterization of absolute retracts of planar graphs is complete:

\begin{theorem}[\cite{Hel74a}]\label{thm:absolute-planar-retract}
A planar graph $G$ is an absolute retract of planar graphs if and only if it is maximal planar and, in an embedding of $G$ in the plane, any triangle bounding a face of $G$ belongs to a subgraph of $G$ isomorphic to $K_4$.
\end{theorem}

To our best knowledge, there has been no relation uncovered between the absolute retracts of planar graphs and other important planar graph subclasses. We make a first step in this direction. Specifically, the Apollonian networks are a subclass of planar graphs that can be defined recursively, as follows.
The triangle $K_3$ is an Apollonian network.
If $G$ is an Apollonian network, and $f$ a triangular face in a plane embedding of $G$, then the graph $G'$, obtained by adding a new vertex adjacent to the three ends of $f$, is also an Apollonian network.
The Apollonian networks of order at least four (all Apollonian networks but the triangle) have some important alternative characterizations, namely they are exactly the maximal planar chordal graphs of order at least four~\cite{MJP06}, the planar $3$-trees~\cite{BiV13}, and the uniquely $4$-colorable planar graphs~\cite{Fow98}.

\begin{proposition}\label{prop:apollonian}
Every Apollonian network with at least four vertices is an absolute retract of planar graphs.
\end{proposition}

\begin{proof}
Let $G$ be an Apollonian network of order at least four. In particular, $G$ is maximal planar. Consider any triangular face $f$ in a plane embedding of $G$. We also have that $G$ is a $3$-tree, and so, all its maximal cliques have four vertices. In particular, there exists a vertex adjacent to all three vertices of $f$. By Theorem~\ref{thm:absolute-planar-retract}, $G$ is an absolute retract of planar graphs. 
\end{proof}

However, this above inclusion is strict ({\it e.g.}, see Fig.~\ref{fig:large-tw} for a counter-example). Indeed, our main result in this section is as follows:

\begin{figure}[!h]
\center
\includegraphics[width=.3\textwidth]{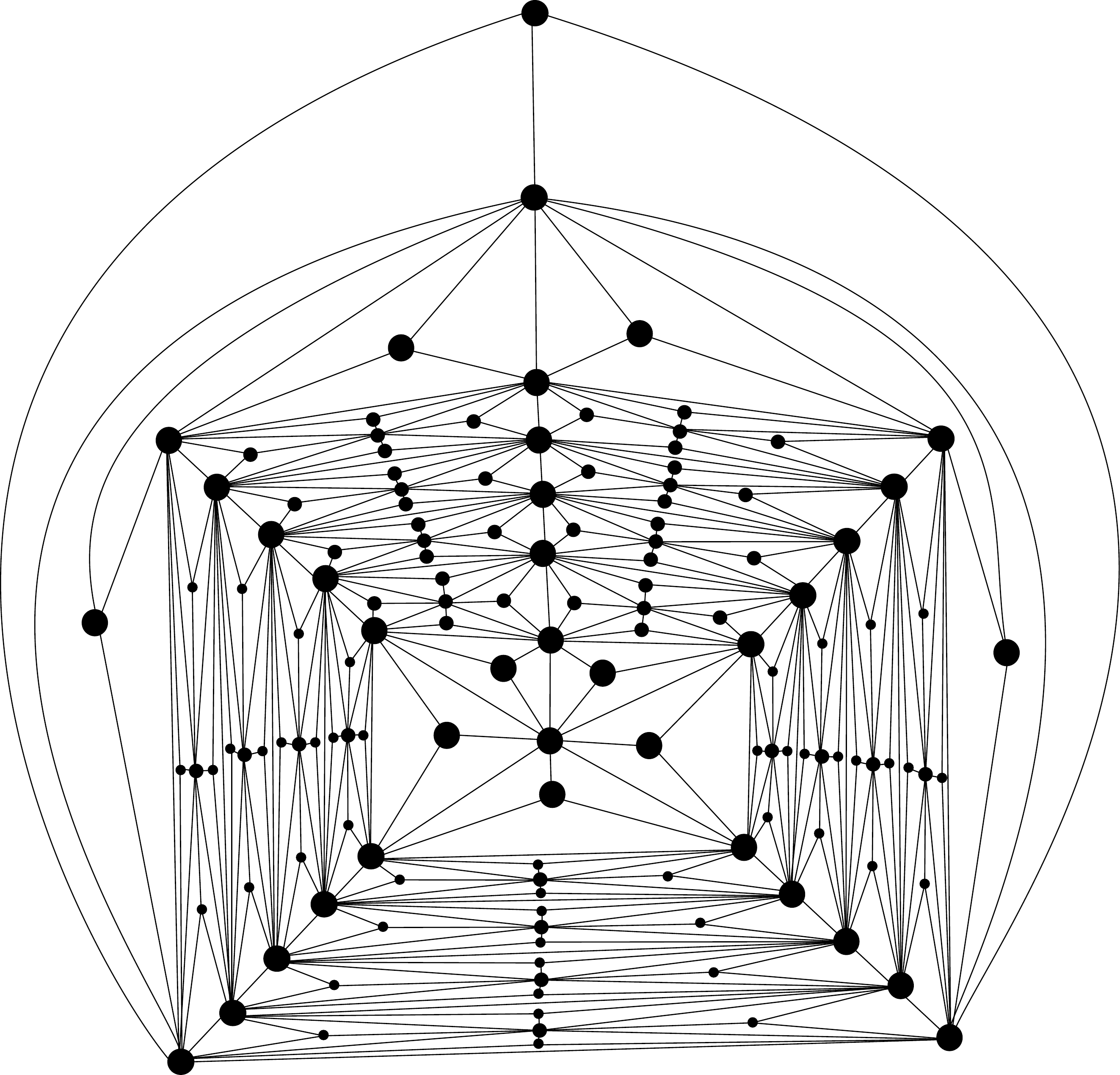}
\caption{An absolute planar retract of treewidth $\geq 5$.}
\label{fig:large-tw}
\end{figure}

\begin{theorem}\label{thm:main-planar}
Every connected planar graph is an isometric subgraph of some absolute planar retract. In particular, there are absolute retracts of planar graphs with arbitrarily large treewidth. 
\end{theorem}

\begin{proof}
Let $G=(V,E)$ be a connected planar graph. The result is trivial if $|V| \leq 2$. From now on, we assume $|V| \geq 3$. We first prove the existence of a biconnected planar graph $H_1$ in which $G$ isometrically embeds. If $G$ is already biconnected, then we set $H_1 := G$. Otherwise, let $u$ be a cut-vertex of $G$. In any fixed planar embedding of $G$, consider a clock-wise ordering of $N_G(u)$. There must be two consecutive neighbours $x,y$ in different connected components of $G \setminus u$. We add a new vertex $z \notin V$ such that $N(z) = \{u,x,y\}$. Since $x,y$ are consecutive in the clock-wise ordering of $N_G(u)$, the resulting graph $G' := G+u$ is still planar. It is also easy to check that $G$ isometrically embeds in $G'$ because $u,x,y$ are pairwise at distance at most two in $G$. Overall, by iterating this above operation, we obtain the desired biconnected planar graph $H_1$ in which $G$ isometrically embeds. 

Note that because $H_1$ is biconnected, in any plane embedding of $H_1$, the boundary of each face is a simple cycle~\cite[Proposition 4.2.6]{Die10}.
Then, we claim that $H_1$ is an isometric subgraph of some planar $H_2$ such that, in some plane embedding of $H_2$ all the faces have length at most five. In order to prove the claim, fix a plane embedding of $H_1$ and let the simple cycle $C$ be the boundary of some face of length $\geq 6$ (if no such face exists then, we are done by setting $H_2 := H_1$). Pick some path $[u,v,w]$ of length two on the cycle $C$. Let $C'$ be obtained from $C$ by contracting $[u,v,w]$ to a single vertex $x$. We add a copy of $C'$ in the face of which $C$ is a boundary, we add an edge between $x$ and all of $u,v,w$, and finally we add an edge between every $y \in C \setminus \{u,v,w\}$ and its copy in $C'$. Note that doing so, we replace the face bounded by $C$ by a new face bounded by $C'$ and by some new triangular and quadrangular faces. Let $H_1'$ be the resulting planar graph. We prove as an intermediate subclaim that $H_1$ isometrically embeds in $H_1'$. For that, fix a shortest-path in $H_1'$ between some two vertices $a,b \in V(H_1)$. Assume there is a subpath $P'=[c_1',\ldots,c_p']$ fully in $C'$ (otherwise, we are done). There is a natural mapping between the latter and the path $P$ of $C$ that is obtained by uncontracting the path $[u,v,w]$. Note that the length of $P$ is at most the length of $P'$ plus two. Furthermore, let $x,y \in C$ be the respective neighbours of $c_1',c_p'$ on the shortest $ab$-path. By construction, $x,y \in V(P)$. Hence, we may replace $[x,P',y]$ by $P$ on the shortest $ab$-path. By repeating this operation until there is no more vertex of $C'$ on the shortest $ab$-path, we prove as claimed that $H_1$ isometrically embeds in $H_1'$. Overall, we repeat the above operations until there is no more face of length $\geq 6$ in the plane embedding. Doing so, we obtain $H_2$.

We further claim that $H_2$ is an isometric subgraph of some plane triangulation $H_3$. This is easily achieved by adding, in the middle of each face, a new vertex adjacent to all the vertices of the boundary. Note that $H_2$ isometrically embeds in $H_3$ because all its faces have length at most five, and so, weak diameter at most two. 

Finally, let us prove that every plane triangulation $H_3$ isometrically embeds in some absolute planar retract $H_4$. The construction is the same as in the previous step, namely, we add in the middle of each (triangular) face a new vertex adjacent to all the vertices of the boundary. Doing so, since each face of $H_4$ shares exactly one edge with a triangular face of $H_3$, we satisfy the following property: any triangle bounding a face of $H_4$ belongs to a subgraph isomorphic to $K_4$. Furthermore, $H_4$ is a plane triangulation of order $\geq |V| \geq 3$, and so, it is maximal planar~\cite[Propositions 4.2.8 and 4.4.1]{Die10}. By Theorem~\ref{thm:absolute-planar-retract}, $H_4$ is an absolute planar retract.
\end{proof}

We stress that the proof of Theorem~\ref{thm:main-planar} is constructive, and that it leads to a polynomial-time algorithm in order to construct an absolute planar retract in which the input planar graph $G$ isometrically embeds. In contrast to our result, the smallest Helly graph in which a graph $G$ isometrically embeds may be exponential in its size~\cite{GDL20}.

\bibliographystyle{abbrv}
\bibliography{biblio}

\end{document}